\documentclass[a4paper,12pt,oneside]{article}

\usepackage{graphicx}
\usepackage[all]{xy}
\usepackage[centertags]{amsmath}
\usepackage{latexsym}
\usepackage{xcolor}
\usepackage{amsfonts}

\usepackage{amssymb,amsthm}
\usepackage[english]{babel}

\usepackage{newlfont}

\usepackage{indentfirst}
\usepackage[english]{babel}
\usepackage[latin1]{inputenc}
\usepackage{newlfont}

\usepackage{indentfirst}
\usepackage[english]{babel}
\usepackage[latin1]{inputenc}
\usepackage{eufrak}
\usepackage{hyperref}

\newtheorem{remark}{Remark}[section]

\theoremstyle{plain}
\newtheorem{lemma}{Lemma}[section]

\newtheorem{proposition}{Proposition}[section]
\newtheorem{theorem}{Theorem}[section]
\newtheorem{definition}{Definition}[section]
\newtheorem{corollary}{Corollary}[section]

\newtheorem{example}{Example}[section]

\newcommand{\s}{\mathfrak s}

\newcommand{\beqn}{\begin{eqnarray}}
\newcommand{\eeqn}{\end{eqnarray}}
\newcommand{\beq}{\begin{eqnarray}}
\newcommand{\eeq}{\end{eqnarray}}

\newcommand{\bpro}{\begin{proposition}}

\newcommand{\epro}{\end{proposition}}
\newcommand{\blem}{\begin{lemma}}
\newcommand{\elem}{\end{lemma}}
\newcommand{\bdfn}{\begin{definition}}
\newcommand{\edfn}{\end{definition}}
\newcommand{\bcor}{\begin{corollary}}
\newcommand{\ecor}{\end{corollary}}
\newcommand{\bthm}{\begin{theorem}}
\newcommand{\ethm}{\end{theorem}}
\newcommand{\bex}{\begin{example}}
\newcommand{\eex}{\end{example}}
\newcommand{\brmq}{\begin{remark}}
\newcommand{\ermq}{\end{remark}}
\newcommand{\benum}{\begin{enumerate}}
\newcommand{\eenum}{\end{enumerate}}
\newcommand{\bitem}{\begin{itemize}}
\newcommand{\eitem}{\end{itemize}}

\theoremstyle{plain}

\usepackage[latin1]{inputenc}

\title{On systems of commuting matrices, Frobenius Lie algebras and Gerstenhaber's Theorem}
  \author{ Andr\'e Diatta$^{( 1)}$ 
;  Bakary Manga$^{( 2)}$
 and  
Ameth Mbaye$^{( 2)}$  \footnote{ \footnotesize \noindent  (1) Aix-Marseille Univ, CNRS, Centrale Marseille, Institut Fresnel, 13013 Marseille, France.
\newline Email: andre.diatta@fresnel.fr; andrediatta@gmail.com.
\newline
 (2) D\'epartement de Math\'ematiques et Informatique,
Universit\'e Cheikh Anta Diop de Dakar,
BP 5005 Dakar-Fann, Dakar, S\'en\'egal. Email: bakary.manga@ucad.edu.sn; ameth3.mbaye@ucad.edu.sn; sombaye100@yahoo.fr
} 
  }  

\begin{document}

\maketitle

\begin{abstract} This work relates to three problems, the classification of maximal Abelian subalgebras (MASAs) of  the  Lie algebra of square matrices, the classification of $2$-step solvable Frobenius Lie algebras and the Gerstenhaber's Theorem.
Let $M$ and $N$ be two commuting square matrices of order $n$ 
with entries in an algebraically closed field $\mathbb K$. Then the associative commutative
$\mathbb K$-algebra, they generate, is of dimension at most $n.$ 
This result was proved by Murray Gerstenhaber in 1961.
The analog of this property for three commuting matrices is still an open problem, 
its version for a higher number of commuting matrices is not true in general.
In the present paper, we give a sufficient condition for this property to be satisfied,
for any number of commuting matrices and arbitrary field $\mathbb K$. Such a result is derived from a discussion on the structure of 2-step solvable Frobenius Lie algebras and a complete characterization of their associated left symmetric algebra structure.
We discuss the classification of $2$-step solvable Frobenius Lie algebras and show that it is equivalent to that of n-dimensional MASAs of the  Lie algebra of square matrices, admitting an open orbit for the contragradient action associated to the multiplication of matrices and vectors.  Numerous examples are discussed in any dimension and a  complete classification list is supplied in low dimensions. Furthermore, in any finite dimension, we give a full classification of all $2$-step solvable Frobenius Lie algebras corresponding to nonderogatory matrices. 
\end{abstract}

\section{Introduction}
Let  $\mathbb K$ be a field, considered here to have characteristic zero. In the $\mathbb K$-vector space $\mathcal M(n,\mathbb K)$ of all $n\times n$ matrices with entries in $\mathbb K,$ we consider the following two algebraic structures: (1) the $\mathbb K$-algebra structure given by the ordinary multiplication and addition of matrices and (2) the Lie algebra structure corresponding to the Lie bracket $[,]$ given by the commutator $[M,N]:=MN-NM$ of matrices $M,N\in \mathcal M(n,\mathbb K)$. We will let $\mathfrak{gl}(n,\mathbb K)$ stand for $\mathcal M(n,\mathbb K)$ together with the latter Lie algebra structure.
Let $\{M_1,\dots, M_k\}$ be a set of $k$ pairwise commuting elements of  $\mathcal M(n,\mathbb K)$. On the one hand, we consider the commutative associative $\mathbb K$-subalgebra $\mathbb K[M_1,\dots, M_k]$  of  $\mathcal M(n,\mathbb K)$ made of all matrices of the form $p_k(M_1,\dots,M_k)$, for all polynomials $p_k(X_1,\dots,X_k)$ in $k$ variables. On the other hand, we denote by $\mathfrak{B}(M_1,\dots,M_k)$  the  vector subspace of   $\mathcal M(n,\mathbb K)$ spanned by $\{M_1,\dots, M_k\}$. If the $M_i$'s are linearly independent, it obviously simply reads $\mathfrak{B}(M_1,\dots,M_k)=\mathbb KM_1\oplus\dots\oplus \mathbb KM_k$, where $\mathbb K M_i$ is the line of matrices of the form $\lambda M_i$, with $\lambda\in\mathbb K$.
 Of course, $\mathfrak{B}(M_1,\dots,M_k)$ is an Abelian Lie subalgebra of $\mathfrak{gl}(n,\mathbb K)$, but in general it is not a $\mathbb K$-subalgebra of 
$\mathcal M(n,\mathbb K)$, although it is a vector subspace of $\mathbb K[M_1,\dots, M_k]$.
In \cite{gerstenhaber-commuting-matrices}, Gerstenhaber proved the following result, hereafter named the Gerstenhaber's theorem. 
If  $M$ and $N$ are two commuting elements of  $\mathcal M(n,\mathbb K)$,  the vector space underlying  $\mathbb K[M,\;N]$ is of dimension at most $n,$ when $\mathbb K$ is algebraically closed.   
The version for  four or a higher number of commuting matrices is not true in general, as shown by Gerstenhaber himself in \cite{gerstenhaber-commuting-matrices}  using the following example. Let $E_{i,j}$ be the square matrix whose coefficients are all zero except the $(i,j)$ entry which is equal to $1.$ One sees that for $n=4$, the system $(E_{1,3},\;  E_{1,4}, \; E_{2,3},\; E_{2,4})$ generates a $5$-dimensional  $\mathbb K$-algebra with $(\mathbb I_{\mathbb K^4},\; E_{1,3},\;  E_{1,4}, \; E_{2,3},\; E_{2,4})$ as a basis, where $\mathbb I_{\mathbb K^n}$  hereafter stands for the identity matrix of order $n$.
The analog of this property for three commuting matrices is still an open problem. It has been proved true up to order  $n=10$, in
 \cite{guralnick-sethuraman,guralnick92,han,holbrook2001,sivic3,sivic2,sivic1}. In \cite{neubauer-sethuraman,commuting-3-matrices}, amongst other results, the authors supply cases where the analog of Gerstenhaber's theorem  for  three commuting matrices holds true. A lot of other works have been dedicated to the subject and its applications amongst which \cite{barria-halmos,kreuzer-robbiano,holbrook-omeara,kandic-sivic,Wadsworth}.  In \cite{decidable}, it is suggested that the problem might have a negative answer in general. But no counterexample has been found so far.
A comprehensive account of the state of the art, up to the year 2011, can be found in \cite{Sethuraman-survey}.

In the present work, we supply a general sufficient condition for the  $k$-matrix analog of Gerstenhaber's theorem to hold true, for any $k\ge 1$.  
Note that in \cite{gerstenhaber-commuting-matrices}, Gerstenhaber further showed that, if  $M$ and $N$ are two commuting elements of  $\mathcal M(n,\mathbb K)$, then  $\mathbb K[M,\;N]$ is always contained in an $n$-dimensional Abelian Lie subalgebra of  $\mathcal M(n,\mathbb K)$. 
 In general   the maximal dimension of a maximal Abelian  subalgebra (MASA, for short) of $\mathcal M(n,\mathbb K)$ is $[\frac{n^2}{4}]+1$, as shown by I. Schur in \cite{schur} for any algebraically closed field $\mathbb K$ and further generalized to any arbitrary field by N. Jacobson \cite{jacobson-schur}.

Our result can be summarized as follows. 
Let $\{M_1,\dots, M_k\}$ be a set of $k$ pairwise commuting elements of  $\mathcal M(n,\mathbb K)$. If one  can complete this set into a set  $\{M_1,\dots, M_k, M_{k+1},\dots, M_{n}\}$ of $n$ linearly independent pairwise commuting matrices, such that the Abelian subalgebra $\mathfrak{B}(M_1,\dots,M_n)$  of $\mathcal M(n,\mathbb K)$, has an open orbit in $(\mathbb K^n)^*$ for the contragredient action  corresponding to the ordinary action of matrices on vectors of $\mathbb K^n,$ then $\dim\mathbb K[M_1,\dots,M_k]\leq n.$ In fact, under such an assumption, the equality $\mathbb K[M_1,\dots,M_n] = \mathfrak{B}(M_1,\dots,M_n)$ holds true, hence any subset of $p$ matrices satisty the Gerstenhaber's theorem, for any $p=1,\dots, n.$ Let us remind here that an action $\rho$ of a Lie algebra $\mathcal G$ on a vector space, say $\mathbb K^n,$ is naturally lifted into an action $\rho^*$ of $\mathcal G$ on the dual space $(\mathbb K^n)^*$ of $\mathbb K^n$, called the corresponding contragredient action given by $\rho^*(a)f = - f \circ\rho(a)$, for any $a\in \mathcal G$ and $f\in(\mathbb K^n)^*$.
The approach so far used is mainly based on the study of the variety of 
$k$-tuples of commuting matrices of $\mathcal M(n,\mathbb K)$. In the case of $3$-tuples of commuting matrices, such a variety is no longer irreducible, so the methods used so far do not apply any longer.
Our approach is different and heavily relies on techniques akin to left invariant affine geometry on Lie groups with a left invariant symplectic structure. 
 It can be summarized as follows. 
We embed $\mathfrak{B}(M_1,\dots,M_n)$ in the Lie algebra of  a $2$-step solvable Lie group with a left invariant exact symplectic structure (Theorem \ref{thm:structure-2-step-Frobenius}) and show that the induced left symmetric algebra (LSA, for short) preserves  $\mathfrak{B}(M_1,\dots,M_n)$ and furthermore it coincides with the ordinary product of matrices (up to a sign) (Theorem \ref{thm:LSA-2-solvable-Frobenius}), hence implying the equality $\mathbb K[M_1,\dots,M_n] = \mathfrak{B}(M_1,\dots,M_n)$.  See Theorem \ref{maintheorem}.
We show that if an $n$-dimensional Abelian subalgebra $\mathfrak{B}$  of  $\mathfrak{gl}(n,\mathbb K)$, has an open orbit on   $(\mathbb K^n)^*,$ for the contragrediente action $\mathfrak{B}\times (\mathbb K^n)^* \to (\mathbb K^n)^*$, $(a,f)\mapsto - f\circ a,$ then $\mathfrak{B}$  is a maximal Abelian subalgebra (MASA) of  $\mathfrak{gl}(n,\mathbb K).$ 
We also discuss the classification of $2$-step solvable Frobenius Lie algebras and further show that it is equivalent to
that of n-dimensional MASAs  of $\mathfrak{gl}(n,\mathbb K)$ admitting an open orbit for the above contragrediente action on  $ (\mathbb K^n)^*.$ See Theorem \ref{theorem:MASAs}.
As applications, several examples are discussed in any dimension, see Examples \ref{example-general-nD}, \ref{example2a}, \ref{example23}. 
The sudy of MASAs has been a vibrant  and vivid subject these last decades, in relation with several subjects in mathematics and physics such as the classification of Lie algebras, the study of dynamical systems, especially the symmetry breaking, complete sets of commuting operators in a 
quantum-mechanical system, the generalized cusps in singularity theory and their applications to computer graphics, ...\cite{dixmier, gerstenhaber-commuting-matrices, holbrook2001, jacobson-schur, schur, sugiura, winternitz}.
A full classification list of all non-isomorphic 2-step solvable Frobenius Lie algebras up to dimension $6,$ is supplied (Theorem \ref{thm:classification-D6}). We fully classify all 2-step solvable Frobenius Lie algebras given by nonderogatory matrices in Theorem \ref{thm:classification}. A new characterization of Cartan subalgebras of $\mathfrak{sl}(n,\mathbb R)$ is given in Theorem \ref{thm:cartansubalgebras}.

The paper is organized as follows. Section \ref{sect:preliminaries},  is devoted to some preliminaries and a few notions.  In Section \ref{sect:2-step-solvable-Frobenius}, 
we give a construction of all 2-step solvable Frobenius Lie algebras and thereby a description of their algebraic structure.  A complete characterization of the corresponding left symmetric algebra (LSA) structure is given in Section \ref{sect:LSA}.  Section \ref{sect:Gerstenhaber-s-theorem} discusses our extension  of  Gerstenhaber's theorem and its proof, together with several examples ain any dimension. In Section \ref{sect:classification}, we  discuss the classification of $2$-step solvable Frobenius Lie algebras.
 The paper ends by 
some concluding remarks in Section \ref{sect:remarks}.

\section{Preliminaries}\label{sect:preliminaries}
Throughout this work, if $\mathfrak{F}$ is a vector space, we will denote its  (linear) dual by $\mathfrak{F}^*$. The symmetric bilinear form $\langle,\rangle$ stands for the 
duality pairing between vectors and linear forms. 
More precisely, for any $u\in\mathfrak{F}$ and $f\in\mathfrak{F}^*$, we define $\langle u,f \rangle$ as $\langle u,f \rangle: = f(u).$
Let $\mathfrak{B}$ be a Lie subalgebra of $\mathfrak{gl}(n,\mathbb K)$. We consider the action $ \mathfrak{B} \times \mathbb K^n \to \mathbb K^n,$ \; $(a,x)\mapsto \rho(a)x:= ax$ of  $\mathfrak{B}$ on $\mathbb K^n$ via the ordinary action of square matrices on vectors.  For any $(a,f)\in  \mathfrak{B} \times (\mathbb K^n)^*$, let  $\rho^*(a)f$ be the element of the dual $(\mathbb K^n)^*$  defined on elements $x$ of $\mathbb K^n$ by
 $\langle\rho^*(a)f, x\rangle:=-\langle f, \rho(a)(x)\rangle= -f(ax).$ We recall that $\rho^*$ is called the contragredient representation (or just action) of $\rho$. 
\begin{definition}  
We say that the orbit of $\alpha\in(\mathbb K^n)^*$, under the action $\rho^*$, is open if it is equal to the whole 
$(\mathbb K^n)^*$. That is $(\mathbb K^n)^*=\{\rho^*(a)\alpha, \;\; a\in \mathfrak{B}\}$. We say that $\alpha$ has a trivial isotropy  if there is no nonzero $a\in\mathfrak{B}$ satisfying $\rho^*(a)\alpha=0.$  
\end{definition}
The vector space $\mathfrak{B}\oplus \mathbb K^n$ endowed with the Lie bracket defined, for every $a,b\in\mathfrak{B}$ and $x,y\in\mathbb K^n$, as 
$[a,b]=[x,y] =0, \;\; \text{ and } \;\; [a,x]=\rho(a)x=ax,$
will be termed the semidirect sum of  $\mathfrak{B}$  and  $\mathbb K^n$  via $\rho$ and denoted by $\mathfrak{B}\ltimes \mathbb K^n$.
A symplectic Lie algebra (also called a quasi Frobenius Lie algebra)  is a Lie algebra $\mathcal G$ together with a nondegenerate skew-symmetric bilinear form $\omega$, which is closed, that is,
 \begin{equation} \label{co-cyle}
 \omega( [u,v],w)+\omega ([v,w],u) +\omega ([w,u],v)=0
, \hskip 2truemm \forall u,v,w\in \mathcal G,
 \end{equation}
where $[,]$ is the Lie bracket of $\mathcal G.$
 Exact symplectic (also termed Frobenius) structures  are special examples of symplectic structures on Lie algebras. They are given by the Chevalley-Eilenberg coboundary $\omega=\partial \alpha$ of  a linear form $\alpha \in\mathcal G^*$, for the adjoint action of $\mathcal G.$  More precisely,  $(\mathcal G,\omega)$ is a Frobenius Lie algebra if there exists $\alpha\in\mathcal G^*$, called a Frobenius functional, such that the skew-symmetric bilinear form $\omega$ defined, for any $u,v\in\mathcal G,$   by
\begin{eqnarray}\label{eq:Frobenius}\omega(u,v)=\partial\alpha(u,v)=-\langle \alpha, [u,v]\rangle,\end{eqnarray}
 is nondegenerate. Every Frobenius Lie algebra is a codimension 1 subalgebra of some contact Lie algebra and could be used to construct the latter.  The converse is true under some conditions  \cite{Diatta-Contact}, \cite{Diatta-Contact-Riem},  \cite{delgado-contact-frobenius-construct}. 

Recall that a left symmetric algebra (LSA) structure on a vector space $\mathfrak{F}$ is a product $\mathfrak{F}\times  \mathfrak{F} \to \mathfrak{F}$, $(u,v)\mapsto u\star v$ such that the corresponding associator $\mathfrak{a}(u,v,w):= (u\star v)\star w-u\star (v\star w)$ is left symmetric, that is,  $\mathfrak{a}(u,v,w) = \mathfrak{a}(v,u,w)$, $\forall u,v,w\in \mathfrak{F}$. Any symplectic structure, say given by a symplectic form $\omega$, on a Lie algebra $\mathcal G$, induces an LSA, denoted hereafter by $\star$, on $\mathcal G$. It is defined,  for any $u,v\in\mathcal G,$  by 
\begin{eqnarray}\label{LSA}
 \omega(u\star v, w):=- \omega(v, [u,w]) . 
\end{eqnarray}
 Moreover, the LSA $\star$ is compatible with the Lie bracket of $\mathcal G$, which is equivalent to the following equality:  for any $u,v\in\mathcal G,$ 
\begin{eqnarray}\label{LSA-torsionfree}
 u\star v-v\star u =[u,v]. 
\end{eqnarray}
Consider the vector space isomorphism $q:\mathcal G\to \mathcal G^*$, $\; u\mapsto i_{u}\omega =\omega(u,\cdot)$. In the case where $\omega=\partial \alpha$, the vector $v_0\in\mathcal G$ such that  $q(v_0) = \alpha$ is called the principal element, corresponding to $\alpha$.

Any Lie group $G$ with Lie algebra $\mathcal G$, has a left invariant  symplectic form $\omega^+$ whose value at the neutral element $\varepsilon$ of $G,$ is precisely $\omega$. Here $\mathcal G$ is identified with the tangent space to $G$ at $\varepsilon.$ 
There  is also, in  $G$, a left invariant affine structure induced by a left invariant  locally flat torsion free connection $\nabla$,  defined  on left invariant vector fields $u^+$, $v^+$  as
\begin{eqnarray}\nabla_{u^+}v^+:=(u\star v)^+, \text{ where } u_\varepsilon^+=u \text{  and  } v_\varepsilon^+=v\end{eqnarray} 
and naturally extended to all  smooth vector fields on $G.$
See for example \cite{diatta-manga}, \cite{diatta-medina}. 

\begin{definition}A Lie algebra $\mathcal G$  is said to be $2$-step solvable if its derived ideal $[\mathcal
G,\mathcal G]$ is Abelian, where $[,]$ stands for its Lie bracket, in other words, 
$[[u,v],[u',v']]=0$  for any $u,v,u',v'$ in $\mathcal G.$ A Lie algebra is said to be indecomposable if it cannot be written 
as the direct sum of two of its ideals. 
\end{definition}

\begin{definition}An Abelian  subalgebra $\mathcal A$ of a  Lie algebra $\mathcal G$  is called  a maximal Abelian subalgebra (MASA) of $\mathcal{G}$  if it is contained in no bigger Abelian subalgebra of $\mathcal{G}$, or equivalently, if $\mathcal A$ coincides with its centralizer $\{b\in \mathcal{G}, [b,a] =0, \; \forall a\in \mathcal A\}$ in $\mathcal G.$ Recall that an $n\times n$ matrix $M$ (resp. linear map) is said to be nonderogatory (or cyclic), if its characteristic and minimal polynomials coincide. \end{definition}
For   $M\in \mathcal M(n,\mathbb K),$ let $\mathbb
K[M]$  and $C(M)$  respectively stand for the ring of polynomials in $M$ with
coefficients in $\mathbb K$ and the space of $n\times n$ matrices (with
coefficients in $\mathbb K$) that commute with $M$.  For a linear map $\psi:\mathbb K^n\to \mathbb K^n,$ we will use the convention $\psi^0=\mathbb I_{ \mathbb K^n}$ and $\psi^{p+1} (x)=\psi( \psi^{p}(x)),$ for every $x\in\mathbb K^n$ and $p$ an integer.
\begin{lemma}[see e.g. \cite{bordemann-medina}, \cite{lucon}] \label{nonderogatorymatrices}
Let $E$ be a vector space of dimension $n$ over a field $\mathbb K$ with
       characteristic zero  and $M$ the matrix of a linear transformation of
$E$, in a given basis of $E$.  Denote $E^*$ the dual space of $E$.
 The following assertions are equivalent:
  \item[1)]   there exists $\bar \alpha\in E^*$ such that $(\bar\alpha,\bar\alpha\circ M,...,
\bar\alpha\circ M^{n-1})$ is a basis of $E^*$,
\item[2)]  there exists $\bar x\in E$ such that $(\bar x,M(\bar x),..., M^{n-1}(\bar x))$ is a basis of
$E$,
\item[3)] $\dim(C(M))=n$ and $C(M)$ is commutative,
\item[4)] $C(M)=\mathbb K[M]$,
\item[5)] the characteristic and the minimal polynomials of $M$
are the same,
\item[6)] In every extension $\bar{ \mathbb K}$  of $\mathbb K$
where $M$ admits a Jordan form, this latter has only one Jordan bloc for
each eigenvalue.
\end{lemma}
 
\section{On the structure of 2-step solvable Frobenius Lie algebras}\label{sect:2-step-solvable-Frobenius}

For simplicity, if  $\{a_1,\dots,a_n\}$ is a set of linearly independent pairwise commuting $n\times n$ matrices, we denote by $\mathfrak{B}$ the vector space 
\begin{eqnarray} \mathfrak{B}=\mathfrak{B}(a_1,\dots,a_n):=\mathbb K a_1\oplus\dots\oplus \mathbb K a_n\end{eqnarray}
 endowed with the Abelian Lie algebra structure $[a_i,a_j]=0$, for any $i,j=1,\dots,n$.
We consider the natural action 
$\mathfrak{B}\times \mathbb K^n\to \mathbb K^n $, $(a,x) \mapsto  \rho(a)x$, given by the ordinary multiplication $\rho(a)x:=ax$ of matrices $a\in\mathfrak{B}$ and vectors $x\in\mathbb K^n$, together with  the corresponding contragredient action
$\mathfrak{B}\times (\mathbb K^n)^*\to (\mathbb K^n)^* $, $(a,f) \mapsto -f\circ a$.

\begin{theorem}\label{thm:structure-2-step-Frobenius}
(A) Let $\{a_1,\dots,a_n\}$ be a set of $n$ linearly independent mutually commuting $n\times n$ matrices with entries in a field $\mathbb K$. Let  $\mathfrak{B}$ stand for  the vector space over $\mathbb K$ spanned by $a_1,\dots,a_n$, looked at as an Abelian Lie subalgebra of $\mathfrak{gl}(n,\mathbb K).$ 
Suppose the action $\mathfrak{B}\times (\mathbb K^n)^*\to (\mathbb K^n)^* $, $(a,f) \mapsto -f\circ a$, has an open orbit in $(\mathbb K^n)^*$. 
 Consider the Lie algebra $\mathcal G$ obtained by performing the semi-direct sum
 $\mathcal G:=\mathfrak{B} {\ltimes} _\rho \mathbb K^n$.  That is, $\mathcal G$ is the vector space $\mathfrak{B}\oplus \mathbb K^n$ endowed with the Lie bracket defined, for every $a,b\in\mathfrak{B}$ and $x,y\in\mathbb K^n$, as 
\begin{eqnarray}[a,b]=[x,y] =0, \;\; \text{ and } \;\; [a,x]=\rho(a)x=ax.\end{eqnarray} 
Then $\mathcal G$ is a 2-step solvable Frobenius Lie algebra.

(B) Conversely, any 2-step solvable Frobenius Lie algebra is isomorphic to a Lie algebra constructed as in (A).
\end{theorem}
\begin{proof}
 (A) Consider the Lie algebra $\mathcal G:=\mathfrak{B} \ltimes_\rho \mathbb K^n$, its underlying Lie bracket reads, for every $a,b\in\mathfrak{B}$ and $x,y\in\mathbb K^n$, as $[a,b]=[x,y] =0$,  and $[a,x]=\rho(a)x:=ax$. By hypothesis, for the contragredient action $\rho^*$, there exists a linear form $\alpha\in(\mathbb K^n)^*$, whose orbit is equal to the whole $(\mathbb K^n)^*.$ So, there exists a basis $(a_1,\dots,a_n)$ of $\mathfrak{B}$ such that $\Big(\rho^*(a_1)\alpha, \; \dots,\; \rho^*(a_n)\alpha\Big)$ is a basis of $(\mathbb K^n)^*.$ 
Now extend $\alpha$  into the element $\tilde \alpha$ of the dual $\mathcal G^*$ of $\mathcal G$ defined by $\tilde\alpha (a)=0$ for any $a\in\mathfrak{B}$ and $\tilde \alpha(x)=\alpha (x),$ for any $x\in\mathbb K^n.$ From now on, we simply write $\alpha$ instead of $\tilde \alpha.$  
Let $(a_1^*,\dots,a_n^*)$ be the dual basis of $(a_1,\dots,a_n)$, so that $\Big(a_1^*,\dots,a_n^*, \rho^*(a_1)\alpha, \; \dots,\; \rho^*(a_n)\alpha\Big)$ is a basis of $\mathcal G^*,$ where each $a_i^*$ is regarded as the element of $\mathcal G^*$ whose restriction to $\mathbb K^n$ is identically equal to zero and which coincides with $a_i^*$ on $\mathfrak{B}$. 
The Chevalley-Eilenberg coboundary $\partial\alpha$, of $\alpha$, is  given, for any $a,b\in\mathfrak{B}$ and any $x,y\in\mathbb K^n$, by
\begin{eqnarray}\label{totallyisotropic}\partial \alpha (a,b):=-\langle\alpha, [a,b]\rangle =0, \;\partial \alpha (x,y):=-\langle\alpha, [x,y]\rangle =0 
\end{eqnarray}
  and
\begin{eqnarray}\partial \alpha (a,x)&:=&-\langle\alpha, [a,x]\rangle = -\displaystyle\sum_{i=1}^n  a_i^*(a) \langle\alpha, [a_i,x]\rangle \nonumber\\
&=&
\displaystyle\sum_{i=1}^n a_i^*(a)\langle\rho^*(a_i)\alpha, x\rangle = \Big( \displaystyle\sum_{i=1}^n a_i^*\wedge \rho^*(a_i)\alpha\Big)(a,x).
\end{eqnarray}
Hence we get 
\begin{eqnarray}\partial \alpha &:=& \displaystyle\sum_{i=1}^n a_i^*\wedge \rho^*(a_i)\alpha
\end{eqnarray}
so that 
\begin{eqnarray}\Big(\partial \alpha\Big)^n &:=& (-1)^{\frac{(n-1)n}{2}}\;n! \; a_1^*\wedge\cdots \wedge a_n^*\wedge \rho^*(a_1)\alpha \wedge\cdots \wedge \rho^*(a_n)\alpha
\end{eqnarray}
which is a volume form on $\mathcal G.$
Hence $\mathcal G$ is an exact symplectic (that is, a Frobenius) Lie algebra. By construction, $\mathcal G$ is 2-step solvable.

(B) Conversely let $\mathcal G$ be a 2-step solvable Frobenius Lie algebra over a field $\mathbb K$ of characteristic $0.$
 As $\mathcal G$ must have  a trivial center, then it splits  as a direct sum of vector  subspaces
$\mathcal G =\mathfrak{B}_1 \oplus \mathfrak{B}_2$ where  $\mathfrak{B}_2$ is the derived ideal $\mathfrak{B}_2:=[\mathcal G,\mathcal
G]$  of $\mathcal G$ and $\mathfrak{B}_1$ is an Abelian Cartan subalgebra  of $\mathcal G$. Set $p_i:= \dim(\mathfrak{B}_i)$, $i=1,2.$
For $a\in \mathfrak{B}_1$ let $\rho(a)$ be the restriction to $\mathfrak{B}_2$ of                the
adjoint $ad(a)$. As the center of $\mathcal G$ is trivial,
$\rho$ then defines a faithful representation of $\mathfrak{B}_1$ by linear
transformations
of $\mathfrak{B}_2$.  In this decomposition, the Lie bracket of $\mathcal G$ reads
\begin{eqnarray} [a,b]=[x,y]=0 \text{ and } [a,x] =\rho(a)x,\end{eqnarray}
  for any  $a,b\in\mathfrak{B}_1$ \ and any  $x,y\in\mathfrak{B}_2.$
A linear form on $\mathcal G$ is closed if and only if
its restriction to $\mathfrak{B}_2$ is identically zero. Let $\rho^*$ be the associated
contragrediente action of $\mathfrak{B}_1$ on $\mathfrak{B}_2^*$ defined by $\rho^* (a)(\alpha):= -
\alpha\circ\rho(a)$, for $a\in \mathfrak{B}_1$ and $\alpha$ in $\mathfrak{B}_2^*$.
 From the decomposition $\mathcal G=\mathfrak{B}_1\oplus \mathfrak{B}_2$, the dual space
 $\mathcal G^*$ of $\mathcal G$ can be viewed as $\mathcal G^*=\mathfrak{B}_2^o\oplus
\mathfrak{B}_1^o$, where $\mathfrak{B}_i^o$ is the  subspace consisting of all linear
forms on $\mathcal G$, whose restriction to $\mathfrak{B}_i$ is identically zero.

Suppose now that $\mathcal G$ is of even dimension $m=2n$, with $n\ge 1$ and
suppose further that $\eta=\alpha + \alpha'$ is a Frobenius functional on
$\mathcal G$, where $\alpha$ is in $\mathfrak{B}_1^o\equiv \mathfrak{B}_2^*$ and
$\alpha'$ in $\mathfrak{B}_2^o\equiv \mathfrak{B}_1^*$. Then for any basis $(a_1,\dots,a_{p_1})$ of
$\mathfrak{B}_{1} $ with dual basis $(a_1^*,\dots,a_{p_1}^*)$, the expression
of $\partial \eta$ reads
 \begin{equation}
 \partial \eta = \partial \alpha = -\displaystyle \sum^{p_1}_{i=1}
 \rho^*(a_i)(\alpha)\wedge a_i^*.
 \end{equation}
 If $k$ is the dimension of the orbit of $\alpha$
under the action $\rho^*$, such a sum merely factorizes into
   \begin{equation}
\partial \eta = \partial \alpha = \displaystyle \sum^{p}_{i=1}
\alpha_i\wedge \beta_i,
  \end{equation}
 where $p=\inf(k,p_1)$, the linear $1$-forms $\alpha_i$ and $\beta_i$
 are respectively in $\mathfrak{B}_{1}^o $ and $\mathfrak{B}_{2}^o $. Due to the property
$(\alpha_i\wedge \beta_i)^2=0$ for each $i=1,\dots,p$, the
$2(p+j)$-form $(\partial \eta)^{p+j}$ is identically zero, if $j$ is greater
than or equal to $1$. But of course $p$ is less than or equal to
$\inf(p_1,p_2)$, as  $k$ is less than or equal to $p_1$.
 Thus as $p_1+p_2=2n$, the non vanishing condition on $(\partial\eta)^n$
 imposes that $n=p=p_1=p_2$. Hence $\rho (\mathfrak{B}_1)$ is  an $n$-dimensional Abelian
subalgebra of the Lie algebra $\mathfrak{gl}(n,\mathbb K)$ of all linear
transformations of $\mathfrak{B}_2$ and
$(\rho^*(a_1)(\alpha),\rho^*(a_2)(\alpha),\dots,\rho^*(a_n)(\alpha))$ is a
basis
of the orbit of $\alpha$, for any basis  $(a_1,\dots,a_n)$ of $\mathfrak{B}_1$. So the orbit of $\alpha$ under $\rho^*$ is open.

We now write $\mathbb K^n=\mathfrak{B}_2=[\mathcal G,\mathcal G]$  so that $\mathcal G=\mathfrak{B}\ltimes_\rho \mathbb K^n$, where $\mathfrak{B}:=\mathfrak{B}_1.$ Fixing a basis on $\mathbb K^n,$ then $\rho(\mathfrak{B})$ can be seen as an $n$-dimensional Abelian Lie algebra of $n\times n$ matrices acting on $\mathbb K^n$ via the canonical action $\rho$ of matrices on vectors, so that, for the corresponding contragredient action, the orbit of $\alpha$ is open. Furthermore, the linear map
 \begin{eqnarray}\psi:\mathcal G=\mathfrak{B}\ltimes_\rho \mathbb K^n\to \rho(\mathfrak{B})\ltimes \mathbb K^n, \;\; \psi(a,x)=\Big(\rho(a),x\Big),\nonumber
\end{eqnarray}
  is an isomorphism between the Lie algebras $\mathcal G$  and $\rho(\mathfrak{B})\ltimes \mathbb K^n$, where the Lie bracket in the latter is the canonical one 
	\begin{eqnarray}\Big[\Big(\rho(a),x\Big), \Big(\rho(b),y\Big)\Big]&=&\Big(\rho(a)\rho(b)-\rho(b)\rho(a), \rho(a)y-\rho(b)x\Big)\nonumber\\
												    & =&\Big(0, \rho(a)y-\rho(b)x\Big). \nonumber\end{eqnarray} 
This can indeed be seen from the equalities
\begin{eqnarray} 
\psi\Big(\Big[(a,x), (b,y)\Big]_{\mathcal G}\Big)&=&\psi\Big(0, \rho(a)y-\rho(b)x\Big)=\Big(0, \rho(a)y-\rho(b)x\Big),
\nonumber\\
\Big[\psi(a,x), \psi(b,y)\Big]_{\psi(\mathcal G)}&=&\Big[\Big(\rho(a),x\Big), \Big(\rho(b),y\Big)\Big]_{\psi(\mathcal G)}=\Big(0, \rho(a)y-\rho(b)x\Big).
\nonumber
\end{eqnarray}
\end{proof}
\begin{remark}\label{rmq:isotropic}Note that the equalities (\ref{totallyisotropic}) mean that $\mathfrak{B}$ and $\mathbb K^n$ are both Lagrangian Abelian Lie subalgebras of $\mathcal G,$ for the exact symplectic form $\omega:=\partial\alpha.$  So $\mathfrak{B}$ and $\mathbb K^n$  integrate to two left invariant foliations, on any 2-step solvable Frobenius Lie group $G$ with Lie algebra $\mathcal G,$ whose leaves are everywhere transverse Lagrangian submanifolds of $G.$ Furthermore, the leaves through the unit $\varepsilon$ are Abelian subgroups of $G$ whose Lie algebras are $\mathfrak{B}$ and $\mathbb K^n$, respectively.
\end{remark}

\section{On the left symmetric algebra induced by the symplectic structure}\label{sect:LSA} 

Let $\mathcal G$ be a $2$-step solvable Frobenius Lie algebra. By a convenient choice of a complementary subspace $\mathfrak{B}$ of 
the derived ideal  $\mathbb K^n=[\mathcal G,\mathcal G]$ of $\mathcal G$, we can always rewrite $\mathcal G$ as the semi-direct sum
$\mathfrak{B}\ltimes \mathbb K^n.$  From Theorem \ref{thm:structure-2-step-Frobenius}, we can chose $\mathfrak{B}$ to be an $n$-dimensional subspace of commuting $n\times n$ matrices. In this section, we  completely characterize the LSA (\ref{LSA}) induced by exact symplectic forms in the case of 2-step solvable Frobenius Lie algebras.
\begin{theorem}\label{thm:LSA-2-solvable-Frobenius}
Let $(\mathcal G=\mathfrak{B}\ltimes \mathbb K^n,\partial \alpha)$ be a $2$-step solvable Frobenius Lie algebra, where $\mathbb K^n$ is the derived ideal 
$\mathbb K^n=[\mathcal G,\mathcal G]$ and $\mathfrak{B}$ a complementary subspace of $\mathbb K^n$. Then, up to a choice of the Frobenius functional $\alpha,$
  the LSA $\star$ in (\ref{LSA}) preserves $\mathfrak{B}$. In other words,
\begin{eqnarray}\label{LSA-preserveA}
a\star b \in\mathfrak{B}, \; \; \text{ for any } a,b\in\mathfrak{B}. 
\end{eqnarray}
Furthermore when $\mathfrak{B}$ is chosen to be an $n$-dimensional subspace of commuting $n\times n$ matrices  with entries in $\mathbb K,$ 
let us denote by $ab$ the ordinary product of the matrices $a, \;b\in\mathcal{M}(n,\mathbb K)$ and by $ax$, the ordinary multiplication of the matrix $a$ and the vector $x\in\mathbb K^n.$ Then the LSA $\star$ in (\ref{LSA}) is completely characterized,  for any $a,b\in\mathfrak{B}$ and any $ x,y\in\mathbb K^n$, by
\begin{eqnarray}\label{eq:homomorphism-LSA}
a\star b=-ab, \;\;  a\star x =0, \;\; x\star y=0, \;\; x\star a = - ax.
\end{eqnarray}
This means that the product $ab$ of two matrices $a$ and $b$ of $\mathfrak{B}$, lies again in $\mathfrak{B}.$ 
The principal element $v_0$ coincides with $-\mathbb I_{\mathbb K^n}$, where $\mathbb I_{\mathbb K^n}$ is the $n\times n$ identity matrix.
\end{theorem}
\begin{proof}
 Set $\omega:=\partial \alpha$.
For any  $a,b\in\mathfrak{B}$, we have 
  $\omega(a\star b,c)=-\omega(b,[a,c])=0$.   Hence we have 
\begin{eqnarray}\label{LSA-preserveB}
a\star b \in\mathfrak{B}, \; \; \text{ for any } a,b\in\mathfrak{B}.
\end{eqnarray}
Indeed, if $a\star b$ were not an element of $\mathfrak{B}$, then there would exist some $c_0\in\mathfrak{B}$ and a nonzero $y_0\in\mathbb K^n$ 
such that  $a\star b=c_0+y_0.$ So for every $c\in\mathfrak{B}$, we  would have $0=\omega(a\star b,c) =\omega(y_0,c) $ and as $\mathbb K^n$ is isotropic
(Remark \ref{rmq:isotropic}), we would also have $\omega(y_0,x) =0$ for every $x\in\mathbb K^n.$ That means that   $\omega(y_0,u) =0$ for every $u\in\mathcal G,$ which is in contradiction with the fact that $y_0$ is nonzero, given that $\omega$ is nondegenerate. Note that, as $\mathfrak{B}$
is Abelian,  Equation (\ref{LSA-torsionfree}) implies that the restriction of $\star$ to $\mathfrak{B}\times\mathfrak{B}$ is commutative
\begin{eqnarray}\label{eq:commutativeLSA}
a\star b= b\star a
\end{eqnarray}
for any $a,b\in\mathfrak{B}$.
Below, we now show that $a\star b$ actually coincides with the matrix multiplication $ab$, up to a sign. In order to do so, we only need to prove that $a\star b$ and $-ab$ coincide when applied on vectors
 $x\in\mathbb K^n.$ Indeed, for any $a,b,c\in\mathfrak{B}$ and any $x\in\mathbb K^n$, we have
 \begin{eqnarray}\omega (c, (a\star b)x) &=& \omega (c, [a\star b, \; x]) = - \omega ((a\star b)\star c, x). \nonumber
\end{eqnarray}
We apply  here the commutativity (\ref{eq:commutativeLSA}) of $\star$ to get $(a\star b)\star c = c\star  (a\star b)= c\star  (b\star a)$, so that
 \begin{eqnarray}
\omega (c, (a\star b)x) &=& -  \omega ( c\star (b\star a), x) =  \omega ( b\star a, [c,x]) 
 = - \omega ( a, [b,[c,x]])\nonumber
\\
&=& \langle \alpha,  [a, [b,[c,x]]]\rangle.\label{eq:intemediate}
\end{eqnarray}
The last equality above  comes from Equation (\ref{eq:Frobenius}), with $u=a$ and $v=[b,[c,x]]$.
Now the Jacobi identity, coupled with the fact that $\mathfrak{B}$ is Abelian,  give the property  $[a, [b,[c,x]]]= [a, [c,[b,x]]] = [c, [a,[b,x]]]$, which we plug into the above equalities (\ref{eq:intemediate}) to get
\begin{eqnarray}
\omega (c, (a\star b)x) &=&  \langle \alpha,  [c, [a,[b,x]]]\rangle  = -\omega(c, [a,[b,x]])= -\omega(c, abx). 
\end{eqnarray}
On the other hand, as $\mathbb K^n$ is totally isotropic with respect to $\omega,$ we obviously have 
\begin{eqnarray}\omega (y, (a\star b)x) =0=\omega(y, abx), \text{ for any } y\in\mathbb K^n.\end{eqnarray}

 Altogether, this is equivalent to the equality  $a\star b=-ab=-ba,$ for any $a,b\in\mathfrak{B}$.
Now using the commutativity of $\mathfrak{B}$ and again the fact that $\mathbb K^n$ is totally isotropic, we get
\begin{eqnarray}
\omega (a\star x, c) &:=&  - \omega ( x, [a,c]))=0, \;\;\;  \omega (a\star x, y) :=  - \omega ( x, ay)=0, 
\end{eqnarray}
which imply the equality $a\star x =  0,$ for any $a\in\mathfrak{B}$ and any $x\in\mathbb K^n.$ In the same way, 
\begin{eqnarray}
\omega (x\star a, b) &:=&  - \omega (a, [x,b]))=\omega (a, [b,x]))= - \omega (b\star a,x)= - \omega (a\star b,x)\nonumber\\
&=&  \omega (b, [a,x])=-\omega (ax, b)
\end{eqnarray}
and 
\begin{eqnarray}
\omega (x\star a, y) &:=&  - \omega (a, [x,y]))= 0 = -\omega (ax, y),
\end{eqnarray}
for any $a,b\in \mathfrak{B}$ and any $x,y\in\mathbb K^n.$ This proves the equality $x\star a=-ax$, for any $a\in\mathfrak{B}$ and any $x\in\mathbb K^n.$ It is also easy to see that $x\star y=0$, due to the equalities
\begin{eqnarray}
\omega (x\star y, a) &:=&   \omega (y, ax)=0, \;\; \text{ and } \omega (x\star y, z)=- \omega (y,[x,z])=0.\nonumber
\end{eqnarray}
The Frobenius functional $\alpha$ can be chosen in $(\mathbb{K}^n)^*,$ so that $\langle\alpha,a\rangle=0$, for any $a$ in $ \mathfrak{B}.$ Consider the vector space isomorphism $q:\mathcal G\to \mathcal G^*$, $\; u\mapsto q(u):=i_{u}\omega =\omega(u,\cdot)$. 
 Let $v_0$ be the principal element, corresponding to $\alpha$, that is,  $q(v_0) = \alpha.$
For every $a\in\mathfrak{B}$, we have $0=\langle\alpha, a\rangle = \omega(v_0,a)$. Hence, using the same reasoning as above, we conclude that $v_0$ is an element of $\mathfrak{B}.$
It is well known that $v_0$ is a right unit, in the general case, for the LSA $\star$ defined in (\ref{LSA}), induced on $\mathcal G$ by $\omega$  (see e.g.\cite{diatta-manga}).
Now we are going to prove that, in the particular case of 2-step solvable Frobenius Lie algebras at hand, $v_0$ coincides with $-\mathbb I_{\mathbb K^n}$.  
Let $x\in\mathbb K^n$. For any $a\in \mathfrak{B},$ we have 
\begin{eqnarray}\omega (a, v_0x)&=& - \omega (v_0\star a,x) = - \omega (a\star v_0,x)\nonumber\\
& =& \omega ( v_0,[a,x]) = \langle \alpha, [a,x]\rangle 
= - \omega (a,x).\end{eqnarray} 
Hence  $v_0 x = -x$, for any $x\in \mathbb K^n$ and  so 
$v_0  = -\mathbb I_{\mathbb K^n}.$ 
\end{proof}

\section{A generalization of Gerstenhaber's theorem}\label{sect:Gerstenhaber-s-theorem}

In the following,  we prove the Gerstenhaber's theorem for  the general case of any set  $\{a_1,\dots,a_k\}$  of $k$ pairwise commuting $n\times n$ matrices,  for any 
$1\leq k \leq n$. This is done under the assumption that such a set can be completed into a set $\{a_1,\dots,a_k,a_{k+1},\dots,a_n\}$ of $n$  linearly independent  $n\times n$ matrices such that, under the canonical action of matrices on vectors of $\mathbb K^n$,  the corresponding Abelian  Lie algebra has an open orbit in the space $(\mathbb K^n)^*$ of linear forms on $\mathbb K^n$.

\begin{theorem}\label{maintheorem}
Let $S:=\{a_1,\dots,a_n\}$ be a set of $n$ linearly independent mutually commuting $n\times n$ matrices with entries in a field $\mathbb K$. Suppose that, as an  $n$-dimensional Abelian Lie subalgebra of $\mathfrak{gl}(n,\mathbb K),$ the vector space  $\mathfrak{B}$ over $\mathbb K$ spanned by $a_1,\dots,a_n$, acts on $(\mathbb K^n)^*$ with an open orbit, via the action $\rho^*: \mathfrak{B}\times (\mathbb K)^*\to  (\mathbb K)^*,$ $(a,f)\mapsto \rho^*(a)f:=-f\circ a.$ 
For any integer $p$, with $1\leq p\leq n$  and for any  subset  $\{a_{i_1},\dots,a_{i_p}\}$ of $p$ elements of  $ S$, denote by $\mathbb K[a_{i_1},\dots,a_{i_p}]$ the commutative associative $\mathbb K$-algebra of all polynomials in  $a_{i_1},\dots,a_{i_p}$, with coefficients in $\mathbb K$. Then we have $\dim\mathbb K[a_{i_1},\dots,a_{i_p}]\le n\;,$  in particular, $\mathbb K[a_{1},\dots,a_{n}] =\mathfrak{B}  \;.$
 \end{theorem}
\begin{proof}
Theorem \ref{maintheorem} is a consequence of Theorem \ref{thm:structure-2-step-Frobenius} and  Theorem \ref{thm:LSA-2-solvable-Frobenius}.  Indeed, if  $\mathfrak S:=\{a_1,\dots,a_n\}$ is a set of $n$ linearly independent mutually commuting $n\times n$ matrices with entries in a field $\mathbb K$ of characteristic zero, we consider the vector subspace  $ \mathfrak{B}:=\mathbb K a_1\oplus\dots\oplus \mathbb K a_n$ of  $\mathcal M(n,\mathbb K)$ and naturally endow it with the structure of an Abelian Lie algebra. First, using Theorem \ref{thm:structure-2-step-Frobenius} we embed $\mathfrak{B}$ as an Abelian  Cartan subalgebra of a 2-step solvable Frobenius Lie algebra $(\mathcal G,\partial\alpha).$  Then we use Theorem \ref{thm:LSA-2-solvable-Frobenius} to show that the LSA $\star$ of  $(\mathcal G,\partial\alpha)$ preserves $\mathfrak{B}$. Better yet, we show that the restriction of $\star$ to $\mathfrak{B}\times \mathfrak{B}$ coincides with the usual product of matrices of $\mathfrak{B}$, up to a sign. In particular, the product $ab$  of any two matrices $a,b\in\mathfrak{B}$, lies again in $\mathfrak{B}$. That is precisely equivalent to the equality
$\mathbb K[a_1,\dots,a_n] = \mathfrak{B},$
or, equivalently,  for any $p=1,\dots,n$ and for every subset $\{a_{i_1},\dots,a_{i_p}\}$ of $p$ elements of  $\mathfrak{S}$, the vector space underlying the $\mathbb K$-algebra
$\mathbb K[a_{i_1},\dots,a_{i_p}]$, is a vector subspace of $\mathfrak{B}$,
which implies the needed inequalities
$\dim\mathbb K[a_{i_1},\dots,a_{i_p}] \leq \dim\mathfrak{B}=n.\nonumber$
\end{proof}

\begin{definition}We say that a system  $S:=\{M_1,\dots,M_p\}$ of matrices is polynomially independent if none of the matrices $M_j\in S$ can be written as a linear combination of polynomials of the remaining elements of $S$. That is $M_j$ is not an element of the algebra $\mathbb K[M_1,\dots,M_{j-1}, M_{j+1},\dots,M_p]$ of polynomials of elements of $S\setminus\{M_j\}$.
The degree of polynomial freedom of the system $S$ or of the algebra  $\mathbb K[M_1,\dots,M_p]$  is the minimum number $q\leq p$ of polynomially independants $M_{j_1},\dots,M_{j_q}\in S$, such that   $\mathbb K[M_1,\dots,M_p] =  \mathbb K[M_{j_1},\dots,M_{j_q}]$.
\end{definition}
\begin{remark}Although its proof has been made simple, thanks to tools from affine geometry, Theorem \ref{maintheorem} is yet powerful in application, as can be seen in Examples \ref{example1-3D},  \ref{example2-8D}, \ref{example-general-nD}. In particular, Theorem \ref{maintheorem} combined with Theorem \ref{theorem:MASAs}, give a receipe consisting of a simple algorithm for constructing MASAs  (resp. maximal Abelian nilpotent subalgebras (MANS)) of dimension $n$ (resp. of dimension $\leq n-1$) of $\mathfrak{gl}(n,\mathbb K)$ or  $\mathfrak{sl}(n,\mathbb K)$ with a given degree of polynomial freedom. 
\end{remark}
\begin{remark}\label{rmq:nonderogatory}
From Lemma \ref{nonderogatorymatrices}, if $M$ is a nonderogatory  $n\times n$ matrix,  then the algebra $\mathfrak{B}:=\mathbb K[M]$ of polynomials in $M$ is an $n$-dimensional MASA of $\mathfrak{gl}(n,\mathbb K)$. The 2-step solvable Lie algebra $\mathcal G_M:=\mathbb K[M]\ltimes \mathbb K^n$ is a Frobenius Lie algebra (Theorem \ref{thm:structure-2-step-Frobenius}).  Such Frobenius  Lie algebras are studied in details and fully classified in Section \ref{chap:classification-nonderogatory}. Obviously, if  $M$ and $N$  are $n\times n$ matrices such that $M=PNP^{-1}$, for some  invertible matrix $P$ then $\mathbb K[M]$ and $\mathbb K[N]$ are conjugate and  $\mathcal G_M$ and $\mathcal G_N$  
are isomorphic  via the linear map  $\xi:\mathcal G_N \to \mathcal G_M$,  $\xi(a)=PaP^{-1}$, $ \xi(x)=Px$, for any $a\in \mathbb K[N]$,
$x\in \mathbb K^n$.
\end{remark}
\subsection{Example \ref{example1-3D}}\label{example1-3D} The commutative $\mathbb R$-subalgebra of $\mathcal{M}(3,\mathbb R)$ spanned by the system of matrices $\{e_2:=E_{1,2} , e_3:=E_{1,3}\}$,
coincides with the 3-dimensional vector space $\mathfrak{B}_{3,1}$ spanned by $\{e_1:=\mathbb I_{\mathbb R^3}$, $e_2$, $e_3\}$ which is also an Abelian Lie subalgebra of $\mathfrak{gl}(3,\mathbb R)$. 
In the canonical basis   $(\tilde e_1, \tilde e_2,\tilde e_3)$ of $\mathbb R^3$, with dual basis $(\tilde e_1^*,\tilde e_2^*,\tilde e_3^*)$, the linear form $\tilde e_1^*$ satisfies $\tilde e_1^*\circ e_1 =\tilde e_1^*$,   $\tilde e_1^*\circ e_2 = \tilde e_2^*$,  $\tilde e_1^*\circ e_3 =\tilde  e_3^*$. Thus $e_1^*$ has an open orbit for the action $\mathfrak{B}_{3,1}\times(\mathbb R^3)^*\to (\mathbb R^3)^*,$ $(a,f)\mapsto -f\circ a$. 
So the system $\{ E_{1,2},E_{1,3}\}$ has been completed into a system $\{e_1, e_2,e_3\}$ of linearly independant matrices satisfying the condition of Theorem \ref{maintheorem}.  
 One sees that $\mathfrak{B}_{3,1}$ does not contain any nonderogatory matrix, as the characteristic polynomial $p_M(\lambda) =(x-\lambda)^3$  of any of its elements $M:=xe_1+ye_2+ze_3$,  is different from its minimal polynomial. Indeed, the equality $(x-M)^2=0$ implies that the minimal polynomial of $M$ divides $(x-\lambda)^2$. More precisely, $\mathfrak{B}_{3,1}$ is not the algebra of polynomials of any nonderogatory matrix. 
 In  the basis $(e_1,e_2,e_3,e_4=\tilde e_1,e_5=\tilde e_2,e_6=\tilde e_3)$, the Lie bracket of the Lie algebra  $\mathcal G_{3,1}:=\mathfrak{B}_{3,1}\ltimes \mathbb R^3$,
 reads  $[e_1,e_{3+j}] = e_{3+j}$, $j=1,\dots,3$, $[e_2,e_5]=e_4$, $[e_3,e_6]=e_4$. So $e_4^*$ is a Frobenius functional, as $\partial e_4^*=-e_1^*\wedge e_4^*-e_2^*\wedge e_5^*-e_3^*\wedge e_6^*$ is a symplectic form on $\mathcal G_{3,1}.$ One also notes that $\mathcal G_{3,1}$ is a 1D extension of the Heisenberg Lie algebra $\mathcal H_{5}$ spanned by $e_2,e_3,e_4,e_5,e_6$  and  $\mathcal H_{5}$ is hence its nilradical.

\subsection{Example \ref{example2-8D}}\label{example2-8D} 
Let $\mathfrak{B}_{4,2}$  be the commutative subalgebra of $\mathfrak{gl}(4,\mathbb R)$ spanned by the $4\times 4$ matrices 
$e_1:=\mathbb I_{\mathbb R^4}$, $e_2=M_1:=E_{1,2}+E_{2,3}$, $e_3:=(M_1)^2=E_{1,3}$, $e_4=M_2:=E_{1,4}$. 
   If we let $(\tilde e_1, \dots,\tilde e_4)$ stand for the canonical basis of $\mathbb R^4$, with dual basis $(\tilde e_1^*,\dots,\tilde e_4^*)$, we clearly see that 
$\tilde e_1^*$ satisfies $\tilde e_1^*\circ e_1 =\tilde e_1^*$,   $\tilde e_1^*\circ e_2 = \tilde e_2^*$,  $\tilde e_1^*\circ e_3 =\tilde  e_3^*$,  $\tilde e_1^*\circ e_4 = \tilde e_4^*$. Thus $e_1^*$ has an open orbit for the action $\mathfrak{B}_{4,2}\times(\mathbb R^4)^*\to (\mathbb R^4)^* $, $(a,f)\mapsto - f\circ a.$ So $\{M_1,M_2\}$ has been completed into a system $\{e_1, e_2,e_3,e_4\}$ of linearly independent matrices satisfying the condition of Theorem \ref{maintheorem}.  
One has  $M_1^3 =M_1M_2=M_2^2=0$. So any polynomials in elements of $\mathfrak{B}_{4,2}$ is a linear combination of $e_1,$ $e_2,$ $e_3,$ $e_4$. Precisely the equality  $\mathfrak{B}_{4,2} =\mathbb R[M_1,M_2]$ holds. 
 Note that 
$\mathfrak{B}_{4,2}$ does not contain any nonderogatory matrix. Indeed the characteristic polynomial of any matrix $M:=xe_1+ye_2+ze_3+te_4$ is $p_M(\lambda) =(x-\lambda)^4$. But we also have $(x-M)^3=0.$ Thus the minimal polynomial of $M$ divides $(x-\lambda)^3$ and is hence different from its characteristic polynomial. So $\mathfrak{B}_{4,2}$ is not the algebra of polynomials of any nonderogatory matrix. In the corresponding    2-step solvable Frobenius Lie algebra 
 $\mathcal G:=\mathfrak{B}_{4,2}\ltimes \mathbb R^4$, with basis $(e_j,e_{4+j}=\tilde e_j,j=1,2,3,4)$, the Lie bracket reads  $[e_1,e_{4+j}] = e_{4+j}$, $j=1,\dots,4$, $[e_2,e_6]=e_5$, $[e_2,e_7]=e_6$, $[e_3,e_7]=e_5$, $[e_4,e_8]=e_5$ and  $\partial e_5^*=-e_1^*\wedge e_5^*-e_2^*\wedge e_6^*-e_3^*\wedge e_7^*-e_4^*\wedge e_8^*$ is a symplectic form on $\mathcal G.$ So $e_5^*$ is a Frobenius functional.
 
\subsection{Example \ref{example-general-nD}}\label{example-general-nD}  Examples \ref{example1-3D} and \ref{example2-8D} generalize to any dimension $n\ge 3$, into a family of ($n-1$) Frobenius Lie algebras $\mathcal G_{n,p}:=\mathfrak{B}_{n,p}\ltimes \mathbb R^n$,  $p=1,\dots,n-1$, all 2-step solvable. The $\mathfrak{B}_{n,p}$'s are pairwise non-isomorphic commutative algebras of polynomials in $n-1, n-2,\dots,2, 1$ matrices, respectively, satisfying the condition of Theorem \ref{maintheorem}.
 Let $n\ge 3$ be an integer. For a given $p$, with $1\leq p\leq n-1,$ set $M_{n,p}:=\displaystyle\sum_{l=1}^{p}E_{l,l+1}$, so that $M_{n,p}^{j-1}=\displaystyle\sum_{i=1}^{p-j+2} E_{i,i+j-1}$,  for  $2\leq j\leq p+1$ and
$M_{n,p}^p=E_{1,p+1}$, and  $\; M_{n,p}^{p+1}=0$. 
Let $\mathfrak{B}_{n,p}$  be the Abelian subalgebra of $\mathfrak{gl}(n,\mathbb R)$ spanned by the matrices $e_j,$ $j=1,\dots,n$, where
$e_j:=M_{n,p}^{j-1}$, $j=1,\dots,p+1$,  and for $j=p+2,\dots,n$, $e_{j}:=E_{1,j}$. In 
 the canonical basis $(\tilde e_1, \dots,\tilde e_n)$ of $\mathbb R^n$, with dual basis $(\tilde e_1^*,\dots,\tilde e_n^*)$, we have
\begin{eqnarray}M_{n,p}\tilde e_1&=&M_{n,p}\tilde e_{q}=0,\;\; q\ge p+2, \nonumber\\
M_{n,p}^{k}\tilde e_q&=&M_{n,p}^{k-1}\tilde e_{q-1}=M_{n,p}^{k-j}\tilde e_{q-j}, \;\;  1\leq j\leq\text{min}(k,q-1), \forall q, \; 2\leq q\leq p+1.\nonumber
\end{eqnarray}
So if $k\ge q$, then $M_{n,p}^{k}\tilde e_q=0$,  and  if $k\leq q-1\leq p$, then $M_{n,p}^{k}\tilde e_q=\tilde e_{q-k}$.
As one can see, we have 
$\tilde e_1^*\circ e_j= \tilde e_1^*\circ M_p^{j-1} = \tilde e_j^*$, $j=1,\cdots,p+1$ and $\tilde e_1^*\circ e_j=\tilde e_1^*\circ E_{1,j}=\tilde e_j^*$, for $j=p+2,\cdots,n$.
 This shows that $\tilde e_1^*$ has an open orbit for the contragredient action of $\mathfrak{B}_{n,p}$ on $(\mathbb R^n)^*$. So the set  $\{M_{n,p},E_{1,p+2},\dots,E_{1,n}\}$ of $n-p$ matrices, has been completed into a system $S=\{e_1,\dots,e_n\}$ of $n$ linearly independant matrices satisfying the condition of Theorem \ref{maintheorem}. The space of all polynomials in elements of $S$ is thus equal to 
$\mathbb R[M_{n,p},E_{1,p+2},\dots,E_{1,n}] =\mathfrak{B}_{n,p}$, so the degree of polynomial freedom of $S$ is $n-p$. Again, $\mathfrak{B}_{n,p}$ does not contain any nonderogatory matrix, as long as the degree of polynomial freedom is $n-p$ is greater or equal to $2$. Thus, $\mathfrak{B}_{n,p}$ is not the space of polynomials of a nonderogatory matrix, if $p\neq n-1$. In the basis 
$(e_j, e_{n+j}=\tilde e_j,\; j=1,\dots,n)$ the Lie bracket of the  2-step solvable Frobenius Lie algebra  
 $\mathcal G_{n,p}:=\mathfrak{B}_{n,p}\ltimes \mathbb R^n$,  is given by the following table  
\begin{eqnarray}
\; [e_1,e_{n+j}] &=& e_{n+j} \;, \; j=1,\dots, n,\nonumber\\ 
\; [e_p,e_{n+q}] &=&0, \text{ if } p\ge q+1,\;\;  [e_k,e_{n+q}]=e_{n+q-k+1} \text{ if } 1\leq k \leq q\leq p+1,
\nonumber\\ 
\; [e_q,e_{n+q}] &=& e_{n+1},  \text{ for }  q=p+2,\dots,n.
\end{eqnarray}
The form  $e_{n+1}^*$ is a Frobenius functional. More precisely, $\partial e_{n+1}^*=-\displaystyle\sum_{i=1}^{n}e_j^*\wedge e_{n+j}^*$ is a symplectic form on $\mathcal G_{n,p}.$ Note that, for $n\ge 3,$ in the family ($\mathcal G_{n,p}$)$_{1\leq p\leq n-1}$, two Lie algebras $\mathcal G_{n,p}$ and $\mathcal G_{n,q}$ are isomorphic if and only if $p=q$. This is due to the fact that the degree of polynomial degree of $\mathfrak{B}_{n,p}$ is $n-p$ which is obviously different from $n-q$ if $p\neq q$, so  $\mathfrak{B}_{n,p}$ and  $\mathfrak{B}_{n,q}$ cannot be conjugate, unless $p=q$. One may also argue that, although each $\mathcal{G}_{n,p}$ has a codimension $1$ nilradical $\mathcal N_{n,p}$ spanned by $(e_2,\dots,e_{2n})$, the derived ideal $[\mathcal N_{n,p},\mathcal N_{n,p}]$ has dimension $p$ and is spanned by $(e_{n+1},\dots, e_{n+p})$. So whenever $p\neq q$, the nilradicals $\mathcal N_{n,p}$ and $\mathcal N_{n,q}$ are not isomorphic and hence, neither are the Lie algebras $\mathcal G_{n,p}$ and $\mathcal G_{n,q}$.
The family ($\mathcal G_{n,p}$)$_{1\leq p\leq n-1}$, has two special cases. (1) The case $p=n-1$ is the only one where the degree of polynomial freedom is $1$, so $\mathfrak{B}_{n,n-1}$ is the algebra $\mathbb R[M_{n,n-1}]$ of polynomials of the nonderogatory matrix $M_{n,n-1}.$  The 2-step-solvable Lie algebras given by nonderogatory matrices are discussed  in Section \ref{chap:classification-nonderogatory} 
  where a full classication is supplied, among other results. In order to keep the same notations as in Section \ref{chap:classification-nonderogatory}, we  let $\mathfrak{D}_0^n$ stand for $\mathcal G_{n,n-1}.$ 
Note that, although we have considered the case $n\ge 3$, one remarks that if we use  $p=n-1$ when $n=2$, 
 we get the $2\times 2$ nonderogatory real matrix $M_{2,1}:=E_{1,2}$. The Lie bracket of the Lie algebra
$\mathcal G_{2,1} := \mathbb R[M_{2,1}]\ltimes \mathbb R^2$  is expressed in the basis $(e_1:=\mathbb I_{\mathbb R^2},e_2=M_{2,1},e_3:=\tilde e_1,e_4:=\tilde e_2)$ as follows, 
$[e_1,e_3]=e_3$, $[e_1,e_4]=e_4$, $[e_2,e_4]=e_3$. The linear form $e_3^*$ is a Frobenius functional and $\partial e_3^*=-e_1^*\wedge e_3^*-e_2^*\wedge e_4^*$. We also set $\mathfrak{D}_0^2:=\mathcal G_{2,1}$.
(2) As regards the case $p=1$, the algebra $\mathcal B_{n,1}$ has $n-1$ degrees of polynomial freedom. The nilradical $\mathcal N_{n,1}$ of $\mathcal G_{n,1}$ is the $(2n-1)$-dimensional Heisenberg Lie algebra $\mathcal H_{2n-1}$. It is spanned by $(e_{j}:=E_{1,j}, j=2,\dots,n, e_{n+k}, \; k=1,\dots,n)$,  with Lie brackets $[e_j,e_{n+j}] =e_{n+1}$, $j=2,\dots,n.$ Note also that the $n-1$ spaces $L_{n,p}=$span($e_2,e_3,\dots,e_n$)  are all $(n-1)$-
dimensional Abelian subalgebras of $\mathfrak{sl}(n,\mathbb R)$, which according to Theorem \ref{theorem:MASAs}, are pairwise non-conjugate MASAs of $\mathfrak{sl}(n,\mathbb R).$

\section{On the classification of 2-step solvable Frobenius Lie algebras}\label{sect:classification}
\subsection{Classification of n-dimensional MASAs and Frobenius Lie algebras}
The following Lemma \ref{lemma:conjugaison-MASAs}  is important in many aspects. It deepens the link between  2-step solvable Lie algebras and MASAs. In general, it is not obvious to tell if two MASAs are conjugate or not.  Invariants (cohomologies, derivations, subalgebras, series of ideals, ...) of the corresponding 2-step solvable Lie algebras may help to tell them apart in a simple way, as we do e.g. in Example  \ref{example-general-nD}. In physics, MASAs are often studied in relation with symmetries of dynamical systems, in this regard Lemma \ref{lemma:conjugaison-MASAs} suggests that invariants of   2-step solvable Lie algebras may encode  informations on such dynamical systems and their governing equations.
\begin{lemma}\label{lemma:conjugaison-MASAs} Let $\mathcal G_1:=\mathfrak B_1\ltimes \mathbb K^n$  and $\mathcal G_2:=\mathfrak B_2\ltimes \mathbb K^n$ be 2-step solvable Lie algebras over $\mathbb K,$ where $\mathfrak B_1$ and $\mathfrak B_2$ are two Abelian subalgebras of $\mathcal M(n,\mathbb K)$.
Then $\mathcal G_1$ and $\mathcal G_2$ are isomorphic if and only if   $\mathfrak B_1$ and $\mathfrak B_2$  are conjugate. That is, if and only if there exists a linear invertible map $\phi:\mathbb K^n\to\mathbb K^n$ such that $\mathfrak B_2=\phi\mathfrak B_1\phi^{-1}$, or equivalently, every element of $\mathfrak B_2$ is of the form $\phi\circ a\circ\phi^{-1} $ for some $a\in\mathfrak B_1$. More precisely,  a linear map $\psi:\mathcal G_1\to\mathcal G_2$ is  a Lie algebra isomorphism if  and only if there exist $x_0\in \mathbb K^n$  and some invertible linear map
 $\phi :\; \mathbb K^n\to\mathbb K^n$ such that $\mathfrak B_2=\phi\mathfrak B_1\phi^{-1}$ and  for any $(a,x)\in \mathfrak B_1\ltimes \mathbb K^n$
\begin{eqnarray}
\psi (a+x) =  \phi\circ a \circ \phi^{-1} + \phi\circ a \circ \phi^{-1}(x_0) +\phi(x) .\end{eqnarray}
\end{lemma}
\begin{proof}Suppose $\mathcal G_1$ and $\mathcal G_2$ are isomorphic under some Lie algebra isomorphism $\psi:\mathcal G_1\to\mathcal G_2.$
As the derived ideal $[\mathcal G_1,\mathcal G_1]=\mathbb K^n$ must be mapped to $[\mathcal G_2,\mathcal G_2]=\mathbb K^n$, 
the components of $\psi$ are linear maps 
$\psi_{1,1}:\mathfrak B_1\to\mathfrak B_2$,  $\psi_{1,2}:\mathfrak B_1\to\mathbb K^n$,  $\phi:\mathbb K^n \to\mathbb K^n$, with $\psi_{1,1}$ and $\phi$ invertible, such that 
 $\psi(a)=\psi_{1,1}(a) +\psi_{1,2} (a)$ and $\psi(x)=\phi(x)$ for any $a\in\mathfrak B_1$ and $x\in\mathbb K^n$. 
We deduce the equality $\psi_{1,1}(a)=\phi \circ a \circ \phi^{-1}$, from the identity $\phi(ax)=\psi([a,x])= [\psi_{1,1} (a)+\psi_{1,2}(a),\phi(x)] = \Big(\psi_{1,1} (a)\circ\phi\Big)(x)$. 
In particular $\phi \circ a \circ \phi^{-1}\in\mathfrak B_2$, for any $a\in\mathfrak B_1$, which is equivalent also to the  equality $\phi \circ \mathfrak B_1 \circ \phi^{-1}= \mathfrak B_2$. 
Taking $b=e_1:=\mathbb I_{\mathbb K^n} \in \mathfrak B_1$ and $x_0:=\psi_{1,2}(e_1)$ in the identity
 $0=\psi([a,b])= [\psi_{1,1} (a)+\psi_{1,2}(a),\psi_{1,1} (b)+\psi_{1,2}(b)]=\psi_{1,1} (a)\psi_{1,2}(b)- \psi_{1,1} (b) \psi_{1,2}(a)$ yields $\psi_{1,2} (a)=\phi\circ a\circ\phi^{-1}x_0$, for any $a\in\mathfrak B_1.$ Note that, $\psi$ being a Lie algebra isomorphism implies $\psi_{1,1} (e_1) =\mathbb I_{\mathbb K^n}$.
Conversely, suppose there exists an invertible linear map $\phi:\mathbb K^n\to\mathbb K^n$ such that $\mathfrak B_2=\phi\mathfrak B_1\phi^{-1}$. Then $\psi:\mathcal G_1\to\mathcal G_2$  given for any $a\in\mathfrak{B}_1$ and any $x\in\mathbb K^n$ by  $\psi (a+x) =  \phi\circ a \circ \phi^{-1}+ \phi\circ a \circ \phi^{-1}(x_0) +\phi(x)  $ is a Lie algebra isomorphism, for any $x_0\in\mathbb K^n.$
\end{proof}
\begin{proposition} The index of polynomial freedom of a commutative $\mathbb K$-algebra $\mathfrak{B}$ of matrices is an invariant of $\mathfrak{B}$. More precisely, if  two commutative algebras of matrices $\mathfrak{B}_1$ and $\mathfrak{B}_2$ are conjugate, or equivalently, if the Lie algebras $\mathfrak{B}_1\ltimes \mathbb K^n$ and $\mathfrak{B}_2\ltimes \mathbb K^n$ are isomorphic, then $\mathfrak{B}_1$ and $\mathfrak{B}_2$ have the same  index of polynomial freedom.
\end{proposition}
\begin{proof}Suppose the Lie algebras $\mathfrak{B}_1\ltimes \mathbb K^n$ and $\mathfrak{B}_2\ltimes \mathbb K^n$ are isomorphic.  
Using Lemma \ref{lemma:conjugaison-MASAs},  let $\phi:\mathbb K^n\to\mathbb K^n$ be a linear invertible map such that $\mathfrak B_2=\phi\mathfrak B_1\phi^{-1}$. A system $(a_1,\dots,a_p)$ of elements of $\mathfrak B_1$ is polynomially independant if and only if $(\phi a_1\phi^{-1},\dots,\phi a_p \phi^{-1})$ is polynomially independant in  $\mathfrak B_2$.
\end{proof}
We have the following.

\begin{theorem}\label{theorem:MASAs}
 Let  $\mathfrak{B}$ be an  $n$-dimensional Abelian Lie subalgebra of $\mathfrak{gl}(n,\mathbb K).$  Write $\mathfrak{B}$ as $\mathfrak{B}=\mathbb I_{\mathbb R^n}\oplus L$, where $L$ is an $(n-1)$-dimensional Abelian subalgebra of $\mathfrak{sl}(n,\mathbb K)$. Suppose that, for 
 the canonical action $\rho$ of  $\mathfrak{B}$ on  $\mathbb K^n$   given by the ordinary multiplication $\rho(a)x:=ax$ of matrices $a\in\mathfrak{B}$ and vectors $x\in\mathbb K^n$, the corresponding  contragrediente action $\rho^*:  \mathfrak{B}\times (\mathbb K^n)^* \to  (\mathbb K^n)^*$,
 $(a,f) \mapsto \rho^*(a) f:=- f\circ a$, has an open orbit.   
Then $\mathfrak{B}$ (resp. $L$) is a maximal Abelian subalgebra (MASA) of $\mathfrak{gl}(n,\mathbb K)$ (resp. of $\mathfrak{sl}(n,\mathbb K)$). \end{theorem}
\begin{proof} Suppose $\rho^*:  \mathfrak{B}\times (\mathbb K^n)^* \to  (\mathbb K^n)^*$,
 $(a,f) \mapsto \rho^*(a) f:=- f\circ a$, has an open orbit and consider some $\alpha\in  (\mathbb K^n)^*$ with an open orbit. Thus, any basis $(a_1,\dots,a_n)$ of $ \mathfrak{B}$ gives rise to a basis  $( \rho^*(a_1)\alpha, \dots, \rho^*(a_n)\alpha)$ of $(\mathbb K^n)^*$ and the (linear) orbital map $Q: \mathfrak{B} \to  (\mathbb K^n)^*$, $Q(a) = \rho^* (a)\alpha$ is an isomorphism between the vector spaces  $\mathfrak{B}$ and $(\mathbb K^n)^*$. Just for convenience, we will let $(\hat e_1,\dots,\hat e_n)$ stand for the basis of $\mathbb K^n$ whose dual basis is $(\hat e_1^*=\rho^*(a_1)\alpha, \dots, \hat e_n^*=\rho^*(a_n)\alpha).$
Suppose $\tilde a\in\mathfrak{gl}(n,\mathbb K)$ is such that $[\tilde a, a] =0$ for any $a\in \mathfrak{B}$ and $\tilde a\neq 0$. Assume $\tilde a$ is not an element of 
 $\mathfrak{B}$, then  $ \mathfrak{\tilde B}:= \mathbb K\tilde a\oplus \mathfrak{B}$ is an $(n+1)$-dimensional Abelian subalgebra of $\mathfrak{gl}(n,\mathbb K)$. Because $\dim \mathfrak{\tilde B} = \dim   (\mathbb K^n)^*+1$, the orbital map $\tilde Q:\mathfrak{\tilde B} \to (\mathbb K^n)^*$,  also given by $\tilde Q(a) = \tilde \rho^* (a)\alpha= -\alpha\circ a$, must have a $1$-dimensional kernel. So, there exists some 
$\tilde b=k\tilde a+a_0\neq 0$, with $k\in\mathbb K$ and $a_0\in \mathfrak{B}$ such that $\tilde \rho^* (\tilde b)\alpha= -\alpha\circ \tilde b =0.$  Since $\tilde \rho^* (a_0)\alpha= \rho^* (a_0)\alpha=Q(a_0) \neq 0$ if $a_0\neq 0$, we must have $k\neq 0$. 
We then must have $ \tilde b x=0$, for any $x\in\mathbb K^n,$ or equivalently $\tilde b =0.$ Indeed, for any $x\in\mathbb K^n$, the expression of $\hat b x$ in the above basis is $\hat b x=\displaystyle\sum_{j=1}^n  \langle \hat e_j^*,\tilde b x\rangle \hat e_j.$ But the components $\langle \hat e_j^*,\tilde b x\rangle $ are all egal to zero for any $j=1,\dots,n$, as 
\begin{eqnarray}\langle \hat e_j^*,\tilde b x\rangle= \langle\rho^* (a_j)\alpha, \tilde b x\rangle =
- \langle\alpha, a_j \tilde b x\rangle  = - \langle\alpha, \tilde b  a_j x\rangle =  \langle\tilde \rho^* (\tilde b) \alpha,   a_j x\rangle =0. \nonumber\end{eqnarray}
Now, the equality $\tilde b=k\tilde a+a_0 =0$,  contradicts the assumption that $\tilde a$ is not in $\mathfrak{B}$. So $\tilde a$ must necessary be in $\mathfrak{B}$. This proves that $\mathfrak{B}$ is a MASA of $\mathfrak{gl}(n,\mathbb K)$. The last claim automatically follows, due to the fact that
a subalgebra $L$ of $\mathfrak{sl}(n,\mathbb K)$ is a MASA of $\mathfrak{sl}(n,\mathbb K)$, if and only if 
$\mathbb I_{\mathbb R^n}\oplus L$ is a MASA of  $\mathfrak{gl}(n,\mathbb K)$.
\end{proof}

\subsubsection{Example \ref{example2a}}\label{example2a} On the space  $\mathfrak B_{n,n}:=L_{n,n}\oplus \mathbb R \mathbb I_{\mathbb R^n}$,  where  $L_{n,n}$ is  the space of matrices $L_{n,n}=\{M:=m_nE_{1,n}+\displaystyle\sum_{j=2}^{n-1}m_j(E_{1,j}+E_{j,n}) \;,\; m_j\in\mathbb R, j=2,\dots,n\}$,
 we set $e_1:=\mathbb I_{\mathbb R^n}$ and $e_j:=E_{1,j}+E_{j,n}$ , $j=2,\dots,n-1$, 
$e_n=E_{1,n}$. Let $(\tilde e_j)$ stand again for the canonical basis of $\mathbb R^n$. One sees that 
$\mathfrak{B}_{n,n}$ is an Abelian subalgebra of $\mathfrak{gl}(n,\mathbb R)$, as any  $M=m_1e_1+\dots+m_ne_n$  and  $A=a_1e_1+\dots+a_ne_n$  in $\mathfrak{B}_{n,n}$ commute
\begin{eqnarray}
[M,A]&=& \displaystyle\sum_{j=2, p=2}^{n-1}m_ja_p[E_{1,j}+E_{j,n}, E_{1,p}+E_{p,n}]  
\nonumber\\
&=&  \displaystyle\sum_{j=2, p=2}^{n-1}m_ja_p(\delta_{j,p} - \delta_{j,p}) E_{1,n} =0.\nonumber\end{eqnarray}
 The characteristic polynomial of any  $M=m_1e_1+\dots+m_ne_n$ is ch($X$)$=(m_1-X)^n$ and
 $ (m_1-M)^{2}=(m_2^2+\dots+m_{n-1}^2)E_{1,n}$ and $ (m_1-M)^{3}=0$. So, up to a scaling, the minimal polynomial of $M$ is $(m_1-X)^3$.  Thus for any $n\ge 4,$ the algebra $\mathfrak B_{n,n}$ contains no nonderogatory matrix. More precisely, $\mathfrak B_{n,n}$ is not the algebra of polynomials in any nonderogatory matrix, for any $n\ge 4$.
We have $\tilde e_1^*\circ e_j= \tilde e_j^*$, $j=1,\dots,n.$ So $\tilde e_1^*$ has an open orbit. 
Hence $\mathfrak B_{n,n}$ satisfies the condition of Theorem \ref{maintheorem} and 
$\mathfrak B_{n,n}$ and $L_{n,n}$ are MASAs of $\mathfrak{gl}(n,\mathbb R)$ and $\mathfrak{sl}(n,\mathbb R)$, respectively, according to Theorem \ref{theorem:MASAs}.
 In  the basis $(e_j, e_{n+j}:=\tilde e_j), j=1,\dots,n,$ of the $2$-step solvable Lie algebra $\mathcal G_{n,n}:=\mathfrak{B}_{n,n}\ltimes \mathbb R^n$, with dual basis $(e_1^*,\dots,e_{2n}^*)$, the non-zero Lie bracket of $\mathcal G_{n,n}$ are 
\begin{eqnarray}
& [e_1,e_{n+j}] =e_{n+j}\;, \; j=1,\dots,n,\nonumber\\
& [e_j,e_{n+j}] =e_{n+1}\;, \; [e_j,e_{2n}] =e_{n+j}\;, \; j=2,\dots,n-1\;, \;\; [e_n,e_{2n}]=e_{n+1}\;,
\end{eqnarray}
so 
we have $\partial e_{n+1}^*=-\displaystyle\sum_{j=1}^ne_j^*\wedge e_{n+j}^*$. In other words,  $\mathcal G_{n,n}$ is a Frobenius Lie algebra and $e_{n+1}^*$ is a Frobenius functional. The  nilradical $\mathcal N_{n,n}:=$span$(e_2,\dots,e_{2n})$ of $\mathcal G_{n,n}$ is of
codimension $1$ and is the semi-direct sum $\mathcal N_{n,n}=\mathbb R e_{2n}\ltimes\Big( \mathcal H_{2n-3}\oplus\mathbb R e_{n}\Big)$ of the line $\mathbb R e_{2n}$ and the Abbena Lie algebra $ \mathcal H_{2n-3}\oplus\mathbb R e_{n}$ where the former acts on the latter by nilpotent derivations. Recall that the $(2n-2)$-dimensional Abbena Lie algebra $ \mathcal H_{2n-3}\oplus\mathbb R e_{n}$ is the direct sum of the line $\mathbb R e_{n}$ and the ($2n-3$)-dimensional Heisenberg Lie algebra  $ \mathcal H_{2n-3}=\text{span}(e_{2},\dots,e_{n-1},e_{n+1},\dots,e_{2n-1})$. The derived ideal $[\mathcal N_{n,n},\mathcal N_{n,n}]$ is ($n-1$)-dimensional and spanned by $(e_{n+1},\dots,e_{2n-1})$. So $\mathcal N_{n,n}$ is isomorphic to none of the nilradicals $\mathcal N_{n,p}$, $1\leq p\leq n-2$ of Example  \ref{example-general-nD}, as $[\mathcal N_{n,p},\mathcal N_{n,p}]$ is $p$-dimensional.  Recall that for $p=n-1$, the Lie algebra $\mathcal G_{n,n-1}$ is obtained from a nonderogatory matrix. Hence, altogether,  $\mathcal G_{n,n}$ is not isomorphic to any of the $n-1$ pairwise non-isomorphic Lie algebras $\mathcal G_{n,p}$ of Example  \ref{example-general-nD}.  
Thus from Theorem \ref{theorem:MASAs}, none of the MASAs $\mathfrak{B}_{n,p}$  of Example  \ref{example-general-nD} is conjugate to $\mathfrak{B}_{n,n}$.
\subsubsection{Example \ref{example23}} \label{example23} Consider the following space $L_{n,n}'$ of $n\times n$ matrices 
\begin{eqnarray}L_{n,n}'=\{M:=\displaystyle\sum_{j=2}^{n}m_j(E_{1,j}+E_{n-j+1,n}) \;,\; m_j\in\mathbb R, j=2,\dots,n\}.\end{eqnarray}
On the Abelian algebra $\mathfrak B_{n,n}':=\mathbb R \mathbb I_{\mathbb R^n}\oplus  L_{n,n}'$, 
 we set $e_1:=\mathbb I_{\mathbb R^n},$ $e_j:=E_{1,j}+E_{n-j+1,n},$ $j=2,\dots,n$. The Lie algebra 
$L_{n,n}'$ is Abelian. Indeed, the commutator of  any elements  $M=m_1e_1+\dots+m_ne_n$  and  $A=a_1e_1+\dots+a_ne_n$  of $L_{n,n}'$ is 
\begin{eqnarray}
[M,A]&=& \displaystyle\sum_{j=2, p=2}^{n}m_ja_p[E_{1,j}+E_{n-j+1,n}, E_{1,p}+E_{n-p+1,n}] 
\nonumber\\
&=&  \displaystyle\sum_{j=2, p=2}^nm_ja_p(\delta_{j,n-p+1} - \delta_{p,n-j+1}) E_{1,n} =0.\nonumber\end{eqnarray}
The characteristic polynomial of any element  $M=m_1e_1+\dots+m_ne_n$  of $\mathfrak{B}$ is $\chi(X)=(m_1-X)^n$ and
 $ (m_1-M)^{3}=0$. So, up to a scaling, the minimal polynomial of $M$ is $(m_1-X)^3$. So for any $n\ge 4,$ the algebra $\mathfrak B_{n,n}'$ contains no nonderogatory matrix. More precisely, $\mathfrak B_{n,n}'$ is not the algebra of polynomials in any nonderogatory matrix, for any $n\ge 4$.
We have $\tilde e_1^*\circ e_j= \tilde e_j^*$, $j=1,\dots,n-1,$ $\tilde e_1^*\circ e_n= 2\tilde e_n^*.$ So $\tilde e_1$ has an open orbit. Hence $\mathfrak B_{n,n}'$ and $L_{n,n}'$ are MASAs of $\mathfrak{gl}(n,\mathbb R)$ and $\mathfrak{sl}(n,\mathbb R)$, respectively, according to Theorem \ref{theorem:MASAs}.

 So in the dual basis $(e_1^*,\dots,e_{2n})$  of the basis $(e_j, e_{n+j}:=\tilde e_j), j=1,\dots,n,$ of the $2$-step solvable Lie algebra $\mathcal G_{n,n}':=\mathfrak{B}_{n,n}'\ltimes \mathbb R^n$, the non-zero Lie bracket of $\mathcal G_{n,n}'$ are 
\begin{eqnarray}
& [e_1,e_{n+j}] =e_{n+j}\;, \; j=1,\dots,n,\nonumber\\
& [e_j,e_{n+j}] =e_{n+1}\;, \; [e_j,e_{2n}] =e_{2n-j+1}\;, \; j=2,\dots,n-1\;, \;\; [e_n,e_{2n}]=2e_{n+1}\;,\nonumber\\
\end{eqnarray}
so 
we have $\partial e_{n+1}^*=-\displaystyle\sum_{j=1}^{n-1}e_j^*\wedge e_{n+j}^*-2e_n^*\wedge e_{2n}^*$. In other words,  $\mathcal G_{n,n}'$ is a Frobenius Lie algebra and $e_{n+1}^*$ is a Frobenius functional. 
As in Example \ref{example2a}, the Lie algebra $\mathcal G_{n,n}'$ also has a codimension $1$ nilradical $\mathcal N_{n,n}'$  spanned by $e_2,\dots,e_{2n}$, which is the semi-direct sum $\mathcal N_{n,n}'=\mathbb R e_{2n}\ltimes\Big( \mathcal H_{2n-3}\oplus\mathbb R e_{n}\Big)$ of $\mathbb R e_{2n}$ and $ \mathcal H_{2n-3}\oplus\mathbb R e_{n}$ where the former acts on the latter by nilpotent derivations. So here again,  $\mathcal G_{n,n}'$ is not isomorphic to any of the $n-1$ pairwise non-isomorphic Lie algebras $\mathcal G_{n,p}$ of Example  \ref{example-general-nD} and, from Theorem \ref{theorem:MASAs}, none of the MASAs $\mathfrak{B}_{n,p}$  of Example  \ref{example-general-nD} is conjugate to $\mathfrak{B}_{n,n}'$. We can check that $\mathcal N_{n,n}'$ is not isomorphic to the nilradical  $\mathcal N_{n,n}$ in Example \ref{example2a}, so  $\mathcal G_{n,n}'$ and $\mathcal G_{n,n}$ are not isomorphic, $\mathfrak B_{n,n}'$ and $\mathfrak{B}_{n,n}$ are not conjugate.

\begin{remark}\label{rem:masa-no-frobenius}From  Theorem \ref{theorem:MASAs} and Lemma \ref{lemma:conjugaison-MASAs}, the classification of $2$-step solvable Frobenius Lie algebras is equivalent to that of $n$-dimensional $MASAs$ of $\mathfrak{gl}(n,\mathbb K)$ acting on $(\mathbb K^n)^*$ with an open orbit. However, not all $n$-dimentional MASAs have open orbit on $(\mathbb K^n)^*$. 
Indeed, for $n\ge 3$, the algebra   $\mathfrak{B}_{n}:= \mathbb R \mathbb I_{\mathbb R^n} \oplus L_{n}$, is a MASA of $\mathfrak{gl}(n,\mathbb R)$, where
 $L_{n}:=\Big\{\displaystyle\sum_{i=1}^{n-1}k_{i,n}E_{i,n}\;,\; k _{i,n}\in\mathbb R, i=1,\dots,n-1\Big\}$ is a MASA of $\mathfrak{sl}(n,\mathbb R).$  To see that,
consider $b=\displaystyle\sum_{p,q=1}^nt_{p,q}E_{p,q} \in\mathfrak{gl}(n,\mathbb R)$, with $t_{p,q}\in\mathbb R$, $p,q=1\dots,n.$  We have $[E_{i,n},b]=\displaystyle\sum_{l=1}^{n-1}t_{n,l}E_{i,l} +(t_{n,n}-t_{i,i})E_{i,n} 
-\displaystyle\sum_{1\leq k\leq n, \; k\neq i} t_{k,i}E_{k,n} $, for any $i=1,\dots,n-1.$ 
So the relation $[b,a] =0$, $\forall a\in\mathfrak B_n$, is equivalent to the following, valid for any $i$ with $1\leq i\leq n-1$: $t_{i,i}=t_{n,n}$ and for any $k$ with $1\leq k\leq n$ and $k\neq i$, one has 
  $t_{k,i}=0$. This simply means that for any $1\leq i\leq n-1$, the coefficients of the  i-th column of $b$ are all equal to zero except the (i,i) entry which is equal to $t_{n,n}.$
In other words $b=t_{n,n}\mathbb I_{\mathbb R^n} + t_{1,n}E_{1,n}+\dots + t_{n-1,n}E_{n-1,n}$, or equivalently, $b$ is an element of $\mathfrak{B}_n$. This simply means that $\mathfrak{B}_n$ is a MASA of $\mathfrak{gl}(n,\mathbb R)$.
The orbit  $\{\alpha\circ a, a\in L_{n}\}$  of any $\alpha\in (\mathbb R^n)^*$ is  at most 2-dimensional and spanned by $\alpha$ and $\tilde e_n^*$. More precisely, let $(\tilde e_1,\dots,\tilde e_n)$ be the canonical basis of $\mathbb R^n$ and let $\alpha = s_1\tilde e_1^*+\dots+s_n\tilde e_n^*\in (\mathbb R^n)^*$ , 
where $s_1,\dots,s_n\in\mathbb R,$ then for any $k_1,k_{1,n},\dots,k_{n-1,n}\in \mathbb R$ and $a=k_1\mathbb I_{\mathbb R^3}+k_{1,n}E_{1,n}+\dots+k_{n-1,n}E_{n-1,n}\in \mathfrak{B}_{n}$, one has $\tilde e_i^*\circ a =k_1\tilde e_i^*+k_{i,n}\tilde e_n^*$ ,\;  $i=1,\dots,n-1$ and  $\tilde e_n^*\circ a =k_1\tilde e_n^*$, 
so that
 \begin{eqnarray}\alpha\circ a
&=&
k_1\alpha+(k_{1,n}s_1
+k_{2,n}s_2
+\dots+
 k_{n-1,n} s_{n-1})e_n^*.\nonumber
\end{eqnarray}
For $n=3$, $L_3$ coincides with $L_{2,4}$ in the list of MASA of $\mathfrak{sl}(n,\mathbb R)$ supplied in \cite{winternitz}. 
\end{remark}

\subsection{Cartan subalgebras of $\mathfrak{sl}(n,\mathbb R)$}\label{sect:Cartan-subalgebras-sl(n,R)}
Recall that a Cartan subalgebra $\mathfrak{h}$ of a Lie algebra $\mathcal G$ is a nilpotent subalgebra which is equal to its own normalizer $\mathcal N_{\mathcal G} (\mathfrak{h}):=\{x\in\mathcal G, [x,y]\in\mathfrak{h}, \forall y\in\mathfrak{h}\}$ in $\mathcal G.$
 In other words,  a Cartan subalgebra of $\mathcal G$ is a subalgebra  $\mathfrak{h}$ which is nilpotent and satisfies the condition that if $x\in \mathcal G$ is such that $[x,\mathfrak{h}]\subset \mathfrak{h} $, then $x$ must belong to $\mathfrak{h}$. Cartan subalgebras of semi-simple Lie algebras must necessarily be Abelian, more precisely they are MASAs  which contain only semisimple elements.
 Cartan subalgebras of simple or semi-simple Lie algebras have been extensively studied by several authors amongst which E. Cartan, Harish-Chandra \cite{harish-chandra}, B. Kostant \cite{kostant}, M. Sugiura \cite{sugiura}, etc. It is natural that Theorem \ref{theorem:MASAs} brings them into play in the study of 2-step solvable Frobenius Lie algebras. Recall that a Cartan subalgebra $\mathfrak{h} $ of a semisimple Lie algebra $\mathcal G$, splits into a direct sum $\mathfrak{h} =\mathfrak{h^+} \oplus \mathfrak{h} ^-$ of two subalgebras $\mathfrak{h}^+ $ and $\mathfrak{h}^- $, respectively 
called its toroidal and its vector parts, such that $\mathfrak{h}^+ $ is only made of elements of $\mathfrak{h} $ whose adjoint operator (as a linear operator of $\mathcal G$) only has purely imaginary eigenvalues and adjoint operators of elements of $\mathfrak{h}^- $ have only real eigenvalues. See e.g. \cite{sugiura}. We propose the following characterization of Cartan subalgebras of $\mathfrak{sl}(n,\mathbb R).$
\begin{theorem}\label{thm:cartansubalgebras} Let  $\mathfrak{h}$ be an Abelian subalgebra of $\mathfrak{sl}(n,\mathbb R)$ and set $\mathfrak{B}_{\mathfrak{h}}:=\mathbb R\mathbb I_{\mathbb R^n}\oplus \mathfrak{h}$. The following are equivalent,

\noindent
(a) $\mathfrak{h}$ is a Cartan subalgebra   of $\mathfrak{sl}(n,\mathbb R)$,

\noindent
(b) $\mathfrak{B}_\mathfrak{h}$ is the algebra $\mathbb R[M]$ 
 of  polynomials of  
some nonderogatory $n\times n$ real matrix $M$ with $n$ distinct (real or complex) eigenvalues,

\noindent
(c) the Lie algebra $\mathfrak{B}_\mathfrak{h}\ltimes  \mathbb R^n$ is the direct sum $\mathfrak{aff}(\mathbb C)\oplus\dots\oplus
\mathfrak{aff}(\mathbb C)\oplus \mathfrak{aff}(\mathbb R)\oplus\dots\oplus \mathfrak{aff}(\mathbb R)$ of $p$ copies of the Lie algebra
$\mathfrak{aff}(\mathbb C)$  and $q$ copies of $\mathfrak{aff}(\mathbb R)$, where $p$ and $q$ are the respective  dimensions of the toroidal and the vector parts of $\mathfrak{h}.$  
\end{theorem}
The proof of Theorem \ref{thm:cartansubalgebras} is given in Section \ref{proofofCartan-subalgebras-sl(n,R)}. A complementary characterization is given in Theorem \ref{prop:abeliannilrad}, where the part (c) is also restated and proved.
\subsection{Lagrangian subalgebras of the Heisenberg Lie algebra as MASAs of $\mathfrak{sl}(n,\mathbb R)$}
The MASAs $L_{n,n}$, $L_{n,n}'$ and $L_n$ of  $\mathfrak{sl}(n,\mathbb R)$ in
Example  \ref{example2a}, Example \ref{example23} and Remark \ref{rem:masa-no-frobenius}, are subspaces of the space of $n\times n$  strictly upper triangular matrices spanned by the $n\times n$ matrices $E_{1,j}$,  $E_{j,n}$ $j=2,\dots, n-1$ and $E_{1,n}$. That is, the space of $n\times n$  strictly upper triangular real matrices whose entries are  all equal to zero everywhere, except on the first row and on the last column,  the diagonal being also entirely made of zeros.  
 Note that, such a space together with the commutator $[M,N]:=MN-NM$ of matrices, is a subalgebra of $\mathfrak{sl}(n,\mathbb R)$, which is isomorphic to the Heisenberg Lie algebra $\mathcal H_{2n-3}$ of  dimension $2n-3.$ Indeed, this is the most frequently used representation of  $\mathcal H_{2n-3}$ by $n\times n$ matrices. Recall that, the Heisenberg Lie algebra is also the central extension  $\mathbb R^{2n-4}\times_{\omega_0}\mathbb R\tilde e_{2n-3}$ of the Abelian Lie algebra $\mathbb R^{2n-4}$ along its standard symplectic form $\omega_0(x,y) = \displaystyle\sum_{j=2}^{n-1} (x_jy_{n+j}-x_{n+j}y_{j})$, the Lie bracket being $[x+s\tilde e_{2n-3},y+t\tilde e_{2n+3}]= \omega_0(x,y)\tilde e_{2n-3}$ and the above representation by matrices is obtained by letting  the canonical basis $(\tilde e_2,\dots,\tilde e_{2n-4},\tilde e_{2n-3})$  of $\mathbb R^{2n-3}$ undergo the identification $\tilde e_j\mapsto E_{1,j}$, $\tilde e_{n+j}\mapsto E_{j,n}$, $j=2,\dots, n-1$ and $\tilde e_{2n-3}\mapsto E_{1,n}.$ The 2-form $\omega_0'$ on $\mathcal H_{2n-3}$ such that $\omega_0'(E_{1,j}, E_{1,k}) =\omega_0(\tilde e_j,\tilde e_k)$,
 $\omega_0'(E_{1,j}, E_{k,n}) =\omega_0(\tilde e_j,\tilde e_{n+k})$,  $\omega_0'(E_{j,n}, E_{k,n}) =\omega_0(\tilde e_{n+j},\tilde e_{n+k})$,  $\omega_0'(M, E_{1,n}) =0$ for any $M\in\mathcal H_{2n-3},$ is again denoted by  $\omega_0$.
The Lie bracket of two matrices 
 $M=\displaystyle\sum_{j=2}^{n-1} (m_jE_{1,j}+m_{n+j}E_{j,n})+m_{2n-3}E_{1,n}$ and $N=\displaystyle\sum_{j=2}^{n-1} (n_jE_{1,j}+n_{n+j}E_{j,n})+n_{2n-3}E_{1,n}$, is expressed as
 \begin{eqnarray}
[M,N]&=&\displaystyle\sum_{j=2}^{n-1}\sum_{l=2}^{n-1}(m_jn_{n+l} [E_{1,j},E_{l,n}]+m_{n+j}n_{l} [E_{j,n},E_{1,l}])\nonumber\\
&=&\displaystyle\sum_{j=2}^{n-1}(m_jn_{n+j}-m_{n+j}n_{j})  E_{1,n} =\omega_0(M,N) E_{1,n}.
\end{eqnarray}
So, a subspace of  $\mathcal H_{2n-3}$ is Abelian if and only if it is totally isotropic with respect to $\omega_0$. Thus,  a subspace of
 $\mathfrak{sl}(n,\mathbb R)$   which is included in $\mathcal H_{2n-3}$ as above, is a MASA of  $\mathfrak{sl}(n,\mathbb R)$, if and only if it is a Lagrangian subspace of $\mathcal H_{2n-3}$, that is, a subspace $\mathcal H$ of dimension $n-1$ such that $\omega_0(M,N)=0$, for any $M,N\in\mathcal H$. Thus the classification of the MASAs of   $\mathfrak{sl}(n,\mathbb R)$ which are also subspaces of $\mathcal H_{2n-3}$, is equivalent to the classification of the Lagrangian subspaces of $\mathcal H_{2n-3},$ up to conjugaison.

\section{Classification of 2-step solvable Frobenius Lie algebras given by nonderogatory matrices}\label{chap:classification-nonderogatory}
From Theorems \ref{thm:structure-2-step-Frobenius} and \ref{theorem:MASAs}, every 2-step solvable Frobenius  Lie algebra  is a semidirect sum  $\mathcal G=\mathfrak{B}\ltimes \mathbb R^n$ of two Abelian Lagrangian subalgebras $\mathfrak{B}$ and $\mathbb R^n,$ where 
 $ \mathbb R^n$   is the derived ideal  $\mathbb R^n=[\mathcal G,\mathcal G]$ and $\mathcal A$ is a  maximal Abelian subalgebra of the Lie algebra  $\mathfrak{gl}(n,\mathbb R)$ of $n\times n$ matrices, that acts on $\mathbb R^n$ such that, the corresponding contragredient action has an open orbit. 
Moreover,  $\mathcal A\ltimes \mathbb R^n$  and  $\mathfrak B\ltimes \mathbb R^n$  are isomorphic if and only if   $\mathcal A$ and  $\mathfrak B$ are conjugate. That is, $\mathfrak B=\phi \mathcal A \phi^{-1} $, for some invertible linear map $\phi: \mathbb R^n\to\mathbb R^n$. In this section, we discuss an important family of $2$-step solvable
 Frobenius Lie algebras, those given by
nonderogatory matrices.
As in Remark \ref{rmq:nonderogatory}, when $\mathfrak B$ is the algebra $\mathbb K[M]$ of all the polynomials in a nonderogatory matrix $M$, we will simply write  $\mathcal G_M =\mathbb K[M] \ltimes \mathbb K^n$ instead of $\mathfrak{B} \ltimes \mathbb K^n$. The same holds when we use the setting of nonderogatory linear maps. 

\subsection{Some key examples}\label{sect:examples-nonderogatory}
\subsubsection{The Lie algebra $\mathfrak{D}_0^1:=\mathfrak{aff}(\mathbb R)$}\label{ex:aff(R)}   Consider the simplest nonderogatory linear map $\psi : \mathbb R\to \mathbb R,$ $x\mapsto x.$ The associated 2-step solvable Lie algebra is the Lie algebra $\mathcal G_\psi=\mathfrak{aff}$($\mathbb R$) of the group of affine motions of the real line $\mathbb R.$  It has a basis  $(e_1,e_2)$  in which the Lie bracket reads $[e_1,e_2]=e_2$ and $\partial e_2^*=-e_1^*\wedge e_2^*$ is the unique (up to a constant scalar factor) exact symplectic form, where $(e_1^*,e_2^*)$ is the dual basis of $(e_1,e_2).$
 Of course, here $e_1=\psi^0$ and $e_2$ stands for a basis of the $1-$dimensional vector space $\mathbb R .$ 
We will see (e.g. in Theorem \ref{thm:classification}) that $\mathfrak{aff}(\mathbb R)$ is the  building block that makes up, in a trivial way (direct sum of $n$ copies), the Lie algebra  $\mathcal G_M$, whenever $M$ is a nonderogatory $n\times n$ real matrix with  $n$ distinct real  eigenvalues. We often use the notation $\mathfrak{D}_0^1 :=\mathfrak{aff}(\mathbb R)$.

\subsubsection{The Lie algebra $\mathfrak{D}_{0,1}^2:=\mathfrak{aff}(\mathbb C)$}\label{ex:aff(C)}    Let  $\psi:\mathbb R^2\to\mathbb R^2$ be the linear map with matrix 
$M_{-1}:=E_{2,1}-E_{1,2}$ 
in the canonical basis $(\tilde e_1, \tilde e_2)$,  with dual basis $(\tilde e_1^*,\tilde e_2^*)$. The characteristic and the minimal polynomials of $\psi$ are  both equal to $1+X^2$. So $M_{-1}$ has the  $2$ complex eigenvalues $ i$, $-i$.  Of course, $\mathbb R[M_{-1}] =\mathbb R\mathbb I_{\mathbb R^2}\oplus \mathbb R M_{-1}$ is a MASA of $\mathfrak{gl}(2,\mathbb R)$ and  $\tilde e_1^*\circ (M_0)^0 =\tilde e_1^*$,   $\tilde e_1^*\circ (M_{-1}) = -\tilde e_2^*$. 
 In the basis $(e_1,\cdots,e_4)$ of the  2-step solvable Lie algebra $\mathcal G_{M_{-1}}$, with $ e_1 =  (M_{-1})^0$, $e_2=  M_{-1},$ $e_3:= \tilde e_1$, $e_4:=\tilde e_2$, the Lie bracket  reads
 $[e_1, e_3] = e_3, [e_1, e_4] = e_4, [e_2, e_3] =e_4, [e_2, e_4] = - e_3.$
The exact form  $\partial e_3^*= - e_1^* \wedge e_3^*+ e_2^* \wedge e_4^*$ 
 is a symplectic form on $\mathcal G_{M_{-1}}.$ 
 Note that $\mathcal G_{M_{-1}}$ is the Lie algebra 
$\mathfrak{aff}(\mathbb C)$=$\Big\{\begin{pmatrix} z_1& z_2\\ 0 & 0\end{pmatrix}, z_j=x_j+iy_j, x_j,y_j\in\mathbb R,j=1,2.\Big\}$ of the group of affine motions of the complex line $\mathbb C,$ looked at as a real Lie algebra, with the identifications 
 $e_1=\begin{pmatrix}1& 0\\ 0 & 0\end{pmatrix}$,  $e_2=\begin{pmatrix}i& 0\\ 0 & 0\end{pmatrix},$ 
$e_3=\begin{pmatrix}0& 1\\ 0 & 0\end{pmatrix}$, $e_4=\begin{pmatrix}0& i\\ 0 & 0\end{pmatrix}.$ We set $\mathfrak{D}_{0,1}^2:=\mathfrak{aff}(\mathbb C).$

\subsubsection{The Lie algebras $\mathfrak D_0^n$}\label{ex:D0n}
In the canonical basis $\Big(\tilde e_{1}, \tilde  e_{2}, \cdots, \tilde e_{n}\Big)$   of $\mathbb R^n$, consider the  principal nilpotent matrix
 $M_0= \displaystyle\sum_{i=1}^{n-1}E_{i,i+1}.$
 It is a nonderogatory matrix, as its minimal polynomial  $m_0(X)=X^n$ coincides with its characteristic polynomial. Of course, $0$ is the unique eigenvalue of $M_0$  and it has multiplicity $n$. 
We use the following notation  $e_{1}:=\mathbb I_{\mathbb R^n} = M_0^0,\cdots, e_{j}:=M_0^{j-1},$ $j=1,\dots,n.$
Note that the vector 
 $\bar x:=\tilde e_n$ is such that 
\begin{eqnarray}\label{basisD0n} 
&& e_{n+1}:=\bar x,  \;\;\; \;\; e_{n+2}:=M_0\bar x=\tilde e_{n-1}, \;  \dots, \; 
\nonumber\\
&& e_{n+j}:=M_0^{j-1}\bar x=\tilde e_{n-j+1}, \;\dots, \; \;  e_{2n}:=M_0^{n-1}\bar x=\tilde e_1,\nonumber
\end{eqnarray}
 form a basis of $\mathbb R^n$. In the basis   $\Big(e_{1}, e_{2}, \cdots, e_{2n}\Big)$ of  the Lie algebra $\mathcal G_{M_{0}},$  up to skew-symmetry, the nonzero Lie brackets are
\begin{eqnarray}\label{LieBracketsD0n}
 [e_{i},e_{n+j}]=e_{n+j+i-1}, ~i,j=1,\cdots,n, \end{eqnarray}
where we use the convention $e_{2n+k}=0,$ whenever $k\ge 1.$ In particular, 

$[e_{1},e_{n+j}]=e_{n+j},$ $j=1,\cdots,n,$  and $[e_{j},e_{2n}]=0,$ $j=2,\cdots,n.$
 
The following exact form  on $\mathcal G_{M_{0}}$, is non-degenerate,  
\begin{eqnarray}\partial e_{2n}^*=-\displaystyle\sum_{j=1}^ne_j^*\wedge e_{2n-j+1}^*=- \displaystyle\sum_{j=1}^n e_{n-j+1}^*\wedge e_{n+j}^*\;.\end{eqnarray} Throughout the present paper, we write  ${\mathfrak D}_{0}^{n}$ instead of   $\mathcal G_{M_{0}}.$  The Lie algebra  ${\mathfrak D}_{0}^{n}$ has an $n$-step nilpotent co-dimension $1$ nilradical  $\mathcal N=$span$(e_2, \dots,  e_{2n}).$ Indeed, if we write $C^0(\mathcal N):=\mathcal N$, $C^{p+1}(\mathcal N):=[\mathcal N, C^{p}(\mathcal N)]$, $p\ge 0$, we have $C^{n-1}(\mathcal N) = \mathbb R e_{2n}$ and  $C^{n}(\mathcal N) = 0.$

\subsubsection{The Lie algebra $\mathfrak{D}_{0,1}^n$} \label{ex:D01n}
Consider the real nonderogatory $4\times 4$ matrix $M_{0,1}=E_{2,1}-E_{1,2}+E_{4,3}-E_{3,4}+E_{1,3}+E_{2,4}$. Its  characteristic and minimal polynomials are both equal to $\chi(X)=(X^2+1)^2$. So $i$ and  $-i$  are the only (repeated complex conjugate) eigenvalues of $M_{0,1}$.  So $\mathbb R[M_{0,1}]=$span($M_{0,1}^0,M_{0,1},M_{0,1}^2,M_{0,1}^3$) is a MASA of $\mathfrak{gl}(4,\mathbb R)$.
In the basis  $(e_1,e_2,\dots,e_8)$ of $\mathcal{G}_{M_{0,1}}$, with $e_1:=\mathbb I_{\mathbb R^4}$,  $e_2:=M_{0,1}$, 
$e_3:=M_{0,1}^2$, $e_4:=M_{0,1}^3$, 
 $e_{4+j}=\tilde e_j$ , $j=1,\dots,4$, 
the Lie bracket reads
\begin{eqnarray}\label{eq:forgottenLiealgebra} && [e_1,e_l]=e_l\;, \;\; l=5,6,7,8,\\
&&[e_2,e_5]=e_6\;,\;\;\;\;[e_2,e_6]=-e_5\;,\; [e_2,e_7]=e_5+e_8\; ,\;\;\;\;\; [e_2,e_8]=e_6-e_7\;,
\nonumber\\
&&[e_3,e_5]=-e_5\;, \; [e_3,e_6]=-e_6\;,\; [e_3,e_7]=2e_6-e_7\;,\;\; \;\; [e_3,e_8]=-2e_5 -e_8\;,
\nonumber\\
&&[e_4,e_5]=-e_6\;,\; [e_4,e_6]=e_5\;,\;\; \;\;[e_4,e_7]=-3e_5-e_8\;,\; [e_4,e_8]=-3e_6+e_7\;.
\nonumber
\end{eqnarray}
Note that both $(e_7,M_{0,1}e_7 = e_5+e_8,  
M_{0,1}^2e_7 = 2e_6-e_7, M_{0,1}^3e_7 = -3e_5-e_8)$ and 
$(e_8,M_{0,1}e_8 = e_6-e_7, \; 
 M_{0,1}^2e_8 =- 2e_5-e_8, M_{0,1}^3e_8 = -3e_5+e_7)$ are bases of $\mathbb R^4.$ 
The 2-form $\partial e_5^* =- e_1^*\wedge e_5^* + e_2^*\wedge( e_6^* - e_7^*) +  e_3^*\wedge (e_5^* + 2  e_8^*)   -  e_4^*\wedge (e_6^* -3 e_7^*) $ is nondegenerate, so $e_5^*$ is a Frobenius functional. We set $\mathfrak{D}_{0,1}^4:=\mathcal G_{M_{0,1}}.$ 
 This generalizes (see Section \ref{sect:nondiagonalizableinC}) to a $2n$-dimensional 2-step solvable Frobenius Lie algebra denoted by $\mathfrak{D}_{0,1}^n:=\mathcal G_{M_{0,1}},$  for any even $n\ge 4$, as follows. Let $M_{0,1}$ be the nonderogatory $n\times n$ matrix  \begin{eqnarray} M_{0,1}= M_s + M_n\, \text{ with }
M_s=-\displaystyle\sum_{j=0}^{\frac{n}{2}-1}(E_{2j+1,2j+2}-E_{2j+2,2j+1}), \;\; M_n=\displaystyle\sum_{j=1}^{n-2}E_{j,j+2},\nonumber
\end{eqnarray} 
in the canonical basis $(\tilde e_1,\dots,\tilde e_n)$ of $\mathbb R^n.$
Its  characteristic and minimal polynomials are both equal to $\chi(X)=(X^2+1)^{\frac{n}{2}}$. So $i$ and  $-i$  are the only (${\frac{n}{2}}$ times repeated complex conjugate) eigenvalues of $M_{0,1}$. We say that each of $i$ and $-i$ is of multiplicity ${\frac{n}{2}}$. So $\mathbb R[M_{0,1}]=$span($M_{0,1}^0,M_{0,1},\dots,M_{0,1}^{n-1}$) is a MASA of $\mathfrak{gl}(n,\mathbb R)$.
In the basis  $(e_1,e_2,\dots,e_{2n})$ of $\mathfrak{D}_{0,1}^n$, with  $e_j:=(M_{0,1})^{j-1}$, $e_{n+j}=\tilde e_{j}$, $j=1,\dots, n$, 
the Lie bracket reads
\begin{eqnarray}\label{eq:forgottenLiealgebra} 
&&[e_j,e_{n+l}]=(M_{0,1})^{j-1}\tilde e_l,  \;\; j, l=1,\dots,n.
\end{eqnarray}

\subsection{The classification Theorem}
 Here we concentrate on the case $\mathbb K=\mathbb R.$ 
We summarize our main results of this section in Theorem \ref{thm:classification}
in which we give a complete classification of all $2$-step solvable Frobenius Lie algebras  of the form 
$\mathcal G_M:=\mathbb R[M]\ltimes\mathbb R^n$,  given by  nonderogatory real matrices $M$. In particular, we show that  the Lie algebras $\mathfrak{D}_0^p$,  $\mathfrak{D}_{0,1}^{2p}$, where $p\ge 1$ is an integer, discussed in Section \ref{sect:examples-nonderogatory}, are the  building blocks that make up, in a trivial way (direct sums), the Lie algebras  $\mathcal G_M$, whenever $M$ is a nonderogatory $n\times n$ real matrix. 
  As in Examples \ref{ex:aff(R)},  \ref{ex:aff(C)}, the notations $\mathfrak{D}_0^1:=\mathfrak{aff}(\mathbb R)$ and  $\mathfrak{D}_{0,1}^2:=\mathfrak{aff}(\mathbb C)$, will be implicitly adopted throughout this work.
It is obvious that, if $z,\bar z$ are two complex conjugate eigenvalues of a square matrix $M$, then the polynomial $(X-Re(z))^2+Im(z)^2$ divides the characteristic polynomial $\chi(X)$ of $M$ (Lemma \ref{lemma:factorization}). 
\begin{definition}Let $M$ be an  $n\times n$ real matrix and $\chi(X)$ its characteristic polynomial. We say that the complex eigenvalues $z,\bar z$ are of multiplicity  $m$, if $m$ is the greatest integer such that  $f_z:=\Big( (X-Re(z))^2+Im(z)^2\Big)^m$ is a factor of  $\chi(X)$. In other words, if  $P_X=\Big( (X-Re(z))^2+Im(z)^2\Big)^q$  divides $\chi(X)$ for some integer $q$, then $P_X$  divides $f_z$.
\end{definition}
We call the following, the Factorization Lemma.
\begin{lemma}[Factorization Lemma]\label{lemma:factorization}
Let $M$ be a nonderogatory $n\times n$ real matrix with only two eigenvalues which are both complex and hence conjugate
 $\lambda,$ $\bar\lambda$. Then, the characteristic polynomial of $M$ is $\chi(X)=\Big((Re(\lambda)-X)^2+Im(\lambda)^2\Big)^{\frac{n}{2}}
.$ 
More generally, if  a real $n\times n$ matrix has $p$ distinct real eigenvalues $\lambda_1,\dots,\lambda_p,$  with respective multiplicities $k_1,\dots,k_p$ and $2q$ complex eigenvalues $\lambda_{p+1},\bar\lambda_{p+1},\dots,\lambda_{p+q},\bar\lambda_{p+q}$, with respective multiplicities $k_{p+1},\dots,k_{p+q}$, with $n=k_1+\dots+ k_p+2(k_{p+1}+\dots+k_{p+q})$. 

Then, the characteristic polynomial 
of $M$ coincides with the following  product $\chi(X)=\displaystyle \Pi_{j=1}^{p}(X-\lambda_j)^{k_j}  \Pi_{l=1}^{q} \Big((X-Re(\lambda_{p+l}))^2+Im(\lambda_{p+l})^2\Big)^{k_{p+l}}.$ \end{lemma}
\begin{proof}
Let $M$ be a nonderogatory $n\times n$ real matrix with only two eigenvalues which are both complex and hence conjugate
 $\lambda,$ $\bar\lambda$.  Then obviously  the polynomial $(X-\lambda)(X-\bar\lambda)=(X-Re(\lambda))^2+Im(\lambda)^2$ divides the characteristic polynomial $\chi(X)$ of $M.$ We factorize $\chi(X)$  as $\chi(X) =\Big((X-Re(\lambda))^2+Im(\lambda)^2\Big)P_1(X)$ 
where $P_1(X)$ is  a polynomial of degree $n-2,$ with real coefficients. As $\mathbb C$ is a closed field, $P_1(X)$ admits some complex zeros, bound to be 
 $\lambda,$ $\bar\lambda$ as they are the only zeros of $\chi(X)$, by hypothesis. We re-write  $\chi(X) $ as $\chi(X) =\Big((X-Re(\lambda))^2+Im(\lambda)^2\Big)^2P_2(X)$ 
where $P_2(X)$ is  a polynomial of degree $n-4$, with real coefficients. The result follows by inductively reapplying the same process to $P_2$. The proof of the general case where $M$ is not necessarily nonderogatory, immediately follows by applying the Primary Decomposition Theorem to $\chi(X)$ to reduce the problem to the cases where $M$ admits a unique eigenvalue or only two eigenvalues which are both complex and conjugate, as above.
\end{proof}
 \begin{theorem}\label{thm:classification}
Let $M$ be a  nonderogatory  $n\times n$ real matrix.
Suppose $M$ has $p$ real distinct eigenvalues $\lambda_1,\cdots,\lambda_p$, with respective multiplicities $k_1, ...,k_p$ and $2q$ distinct complex  eigenvalues $z_1,\bar z_1,\cdots, z_q, \bar z_q$ with respective multiplicities $m_1,\dots,m_q$, where $n=k_1+\dots+k_p+2(m_1+\dots+m_q)$.
Then the Lie algebra 
 $\mathcal G_M:=\mathbb R[M]\ltimes \mathbb R^n$  is isomorphic to the direct sum 
 ${\mathfrak D}_{0}^{k_1}\oplus  \cdots\oplus {\mathfrak D}_{0}^{k_p} \oplus {\mathfrak D}_{0,1}^{m_1}\oplus  \cdots\oplus {\mathfrak D}_{0,1}^{m_q}$
of the Lie algebras  ${\mathfrak D}_{0}^{k_1}, $ $\cdots,$ ${\mathfrak D}_{0}^{k_p}$, $ {\mathfrak D}_{0,1}^{m_1},$ $ \cdots,$ ${\mathfrak D}_{0,1}^{m_q}.$

 In particular, let $M$ be a real $n\times n$ nonderogatory matrix,
\begin{itemize}
\item (a) if M admits $n$ distinct real  eigenvalues, then the Lie algebra 
 $\mathcal G_M$  is isomorphic to the direct sum $\mathfrak{aff}(\mathbb R) \oplus \mathfrak{aff}(\mathbb R) \oplus \ldots \oplus~ \mathfrak{aff}(\mathbb R) $  of $n$
 copies of the Lie algebra $\mathfrak{aff}(\mathbb R) $ of the group of  affine motions of the real line $\mathbb R$,

\item (b) if $M$ admits a unique real eigenvalue $\lambda$, then  $\mathcal G_M$ is isomorphic to ${\mathfrak D}_{0}^{n}$,

\item (c) if M admits $n$ distinct complex  eigenvalues, then 
 $\mathcal G_M$  is isomorphic to the direct sum $\mathfrak{aff}(\mathbb C) \oplus \mathfrak{aff}(\mathbb C) \oplus \ldots \oplus~ \mathfrak{aff}(\mathbb C) $  of $n$
 copies of the Lie algebra $\mathfrak{aff}(\mathbb C) $ of the group of  affine motions of the complex line $\mathbb C$, seen as a real Lie algebra,

\item (d) if $M$ admits only  two eigenvalues  which are both complex (nonreal)  and hence conjugate, then  $\mathcal G_M$ is isomorphic to ${\mathfrak D}_{0,1}^{n}$.
\end{itemize}
\end{theorem}
The following is a direct corollary of Theorem \ref{thm:classification}.
\begin{corollary}\label{corollary1} Let $M$ be a nonderogatory $n\times n$ real matrix. 
The  Lie algebra $\mathcal G_M$ is indecomposable if and only if one of the following holds true: 
(a) $M$ has a unique eigenvalue (which is necessarily real), in which case $\mathcal G_M$ is isomorphic to $\mathfrak{D}_0^n$ and is hence completely solvable, or (b)  $M$ has only 2 eigenvalues which are both complex and conjugate, in which case $\mathcal G_M$ is isomorphic to $\mathfrak{D}_{0,1}^n$.  
\end{corollary}

The rest of this section and Sections \ref{Sect:uniqueeigenvalue}, \ref{Sect:complexeigenvalues} are mainly concerned with discussions and the proof of Theorem \ref{thm:classification}.
Lemma \ref{lem:splitting} plays a fundamental role in the classification of Frobenius Lie algebras of the form $\mathcal G_M$, where $M$ is a nonderogatory matrix. It allows us to  split Theorem \ref{thm:classification} into two main cases discussed in Propositions \ref{prop:classification-multiple-eigenvalues}, \ref{prop:classification-multiple-complexeigenvalues}, \ref{prop:isomorphismcomplexeigenvaluesnonsemiimple}, \ref{prop:all-complex-eigenvalues-nondiagonalizable}.
The particular case of  Theorem \ref{thm:classification} (a) is obtained by taking, in the general case, $p=n$, $q=0$ and by identifying $\mathfrak{D}_0^1$ with $\mathfrak{aff}(\mathbb R)$. 
The proof of the particular case of Theorem \ref{thm:classification}  (b) can be directly found in Lemma \ref{lem:uniquerealeigenvalue}, whereas  Theorem \ref{thm:classification}  (c)  and (d) are directly proved in
Propositions \ref{prop:classification-multiple-complexeigenvalues} and \ref{prop:isomorphismcomplexeigenvaluesnonsemiimple}, respectively.

\begin{lemma}\label{lem:splitting}
Let $M$ be a nonderogatory $n\times n$ real matrix. Suppose $\mathbb R^n$ splits as a direct sum
 $\mathbb R^n= \mathcal{E}_1\oplus\mathcal{E}_2$ of  two  subspaces $\mathcal{E}_1$ and $\mathcal{E}_2$ which are both invariant under $M$, that is, the image
$M v_j$ of any vector $v_j\in \mathcal{E}_j$, still lies in $\mathcal{E}_j,$ for any $j=1,2.$
Denote by $M_j$ the restriction of $M$ to $\mathcal{E}_j$ and let $\mathcal G_{M_j}$ 
stand for the Lie algebra
 $\mathcal G_{M_j}:=\mathbb R[M_j]\ltimes \mathcal{E}_j$.
Then each $\mathcal G_{M_j}$, $j=1,2,$ is an ideal of $\mathcal G_{M}$.
 More precisely, 
$\mathcal G_{M}$ splits as the direct sum 
 $\mathcal G_{M}=\mathcal G_{M_1}\oplus \mathcal G_{M_2}$.
\end{lemma}
\begin{proof} Under the assumptions of Lemma \ref{lem:splitting}
, let $M_j$ be the restriction of $M$ to the invariant subspace $\mathcal E_j,$ $j=1,2.$ We extend $M_{j}$ to a linear map $\tilde {M}_{j}$ on 
$\mathbb R^n$ in such a way that $\tilde{M}_{j}$ identically vanishes 
on $\mathcal{E}_{p}$ whenever $p\neq j.$ So we have $[\tilde {M}_{j} , M] =0$ for every $j=1,2.$ 
That is, $\tilde{M}_{j}$ and, in fact, every polynomial in $\tilde M_j$,  $j=1,2,$  are all  polynomials in $M$, 
hence they are all elements of $\mathbb R[M]$. More precisely $\mathbb R[\tilde M_j]$ is a subalgebra of $\mathbb R[M]$. But $\mathcal{E}_{j}$ being 
obviously an ideal of the Abelian Lie algebra $\mathbb R^n$,
we thus see that the Lie algebra $\mathcal  G_{M_{j}}$, identified with   
$\mathcal G_{{\tilde M}_{j}} =\mathbb R[\tilde{M}_{j}]\ltimes \mathcal{E}_{j}$, is a Lie subalgebra of the Lie algebra $\mathcal G_M:=\mathbb R[M]\ltimes \mathbb R^n.$ 
As a matter of fact, each $\mathcal  G_{M_j}$ is an ideal of $\mathcal G_M$. This is due to $\mathcal E_j$ being stable by $M$  and $\tilde M_j$ identically vanishing on $\mathcal E_p,$ whenever $p\neq j,$ in addition to the sums $M= \tilde M_1+\tilde M_2$ and $\mathbb R^n=\mathcal E_1\oplus\mathcal E_2.$ Thus, as the  ideals $\mathcal{G}_{M_j}$ form a direct sum (they only meet at $\{0\}$, unless they are identical), the Lie algebra $\mathcal G$ splits as the direct sum $\mathcal{G}:=\mathcal G_{{\tilde M}_{1}}\oplus \mathcal G_{{\tilde M}_{2}}$.
\end{proof}

Consider the general case where a real matrix $M$ admits $p$ real and $2q$ complex eigenvalues.  We write the characteristic polynomial $\chi_M(X)$ of $M$ as the product 
 $\chi_M(X) = Q_1(X)Q_2(X)$ where $Q_1(X)$ has only complex (nonreal) zeros, whereas  the zeros of  $Q_2(X)$ are all real.
Of course as $Q_1(X)$ and $Q_2(X)$ must be relatively prime, the 
Primary Decomposition Theorem combined with Cayley-Hamilton theorem imply the following
\begin{eqnarray}\mathbb R^n=\ker(\chi_M(M)) = \ker(Q_1(M))\oplus \ker(Q_2(M)). \end{eqnarray} 
As $\ker Q_1(M)$ and $\ker Q_2(M)$  are both stable by $M$, Lemma \ref{lem:splitting} reduces the proof of Theorem \ref{thm:classification} to two cases: the case where all the eigenvalues of $M$ are real and the case where all the eigenvalues of $M$ are complex. 

\subsection{Nonderogatory matrices with only real eigenvalues}\label{Sect:uniqueeigenvalue}
Let $M$ be a  nonderogatory $n\times n$ real matrix with p distinct eigenvalues $\lambda_1, \dots,\lambda_p,$  all of which are real and of respective multiplicity $k_1,\dots, k_p$, with 
$k_1+\dots+ k_p=n$, so that its characteristic polynomial factorizes as
 $\chi_M(X)=\chi_1(X)\chi_2(X)\cdots\chi_p(X)$, 
where $\chi_j(X)=(X-\lambda_j)^{k_j}$, for $j=1,\dots,p$. The polynomials $\chi_j(X)$
 are pairwise relatively prime. Indeed, any divisor of $\chi_j(X)$ is of the form $(X-\lambda_j)^{r_j}$ for some integer $r_j\leq k_j,$
 for every $j=1,\dots,p.$
So any common divisor of $\chi_i(X)$ and $\chi_j(X)$ would simultaneously be of the forms $(X-\lambda_i)^{r_i}$ and 
$(X-\lambda_j)^{r_j}$, so that the equality
 $(X-\lambda_i)^{r_i} = (X-\lambda_j)^{r_j}$ entails $ r_i=r_j=0$ or $\lambda_i=\lambda_j$ and hence $i=j.$
By the 
Primary Decomposition Theorem and Cayley-Hamilton theorem,  we have 
\begin{eqnarray}
\label{eq:split}\mathbb R^n=\ker\chi_M(M) = \ker\chi_1(M)\oplus\ker\chi_2(M)\oplus\cdots\oplus\ker\chi_p(M).\end{eqnarray} 

Of course the subspaces $\mathcal E_j:=\ker(M-\lambda_j)^{k_j},$ $j=1,\dots,p,$ are all stable by $M$, due to the fact that the endomorphisms $M$ and $(M-\lambda_j)^{k_j}$ commute.

\begin{lemma}\label{lem:nonderogatorysubmatrices}The restriction $M_j$ of $M$ to $\mathcal E_j:=\ker(M-\lambda_j)^{k_j}$ is again a nonderogatory endomorphism of $\mathcal E_j$ with a unique eigenvalue,
for every $j=1,\cdots,p.$
\end{lemma}
\begin{proof} The characteristic polynomial of the restriction $M_j$ of $M$ to the subspace $\mathcal E_j:=\ker(M-\lambda_j)^{k_j}$ is 
$  \chi_j(X)=(X-\lambda_j)^{k_j}$. If the minimal polynomial  
of $M_j$, hereafter denoted by $ m_j(X)$, were different from $\chi_j(X)$, then there would exist some integer $r_j$ with $1\leq r_j<k_j,$ such that 
$ m_j(X)=(X-\lambda_j)^{r_j}$. Furthermore, the polynomial 
$T(X) =    \chi_1(X)\chi_2(X)\dots \chi_{j-1}(X) m_j(X) \chi_{j+1}(X) \dots \chi_{p-1}(X)\chi_{p}(X)$ would be of lower degree than $\chi(X)$ and yet  would  
satisfy $T(M)=0$. Thus the minimal polynomial of $M$ would divide $T(M)$ and would hence be different 
from the characteristic polynomial of $M$. This would contradict the fact that $M$ is nonderogatory.
\end{proof}

The above shows  that, if a nonderogatory matrix $M$ has only real eigenvalues, $\lambda_1,\cdots,\lambda_p$
of respective multiplicity $k_1,\dots, k_p$, then  its restriction $M_j$ to each subspace $\mathcal E_j:=\ker(M-\lambda_j)^{k_j}$ 
is again a nonderogatory endomorphism with 
a unique eigenvalue $\lambda_j$ of multiplicity $k_j$ and $\mathbb R^n $ splits as a direct sum $\mathbb R^n = \mathcal E_1\oplus\cdots\oplus \mathcal E_p.$
To conclude that $M$ has a Jordan form,  we now need to show that each $M_j$ has the following well known Jordan form.
\begin{lemma}[Jordan form]\label{lem:uniquerealeigenvalueJordanform}
If a real nonderogatory  $n\times n$ matrix $M$ has a unique eigenvalue $\lambda$ with multiplicity $n$, 
then there exists a basis $(\tilde e_1', \cdots, \tilde e_n')$ of $\mathbb R^n$ in which $M$ has the form
$M_\lambda =\lambda \mathbb I_{\mathbb R^n} + \displaystyle\sum_{i=1}^{n-1}E_{i,i+1}$
\end{lemma}
\begin{proof}
Let $M$ be a real $n\times n$ nonderogatory matrix with a unique real eigenvalue $\lambda$ with multiplicity $n.$   So, up to a sign, its characteristic  polynomial is $ P_M(X)=(X-\lambda)^n=\displaystyle\sum_{i=0}^n \complement_n^i(-1)^{n-i} \lambda^{n-i} X^{i},$  with $\complement_p^q=\frac{p!}{q!(p-q)!}$, for $p\ge q$. 
By Cayley-Hamilton's Theorem 
$M^n=\displaystyle\sum_{i=0}^{n-1} \complement_n^i(-1)^{n-i+1} \lambda^{n-i} M^{i} $.

Choose $\bar x\in\mathbb R^n$ such that   $(\bar e_1,\dots,\bar e_{n})=$$(  M^{n-1}\bar x, \; M^{n-2}\bar x, \; \cdots,  \;  M^{n-j}\bar x, \cdots, \; \bar x)$ is a basis of $\mathbb R^n.$  We have  
$M\bar e_1=M^n\bar x=\displaystyle\sum_{i=0}^{n-1} \complement_n^i(-1)^{n-i+1} \lambda^{n-i} M^{i}\bar x =
\displaystyle\sum_{j=1}^{n} \complement_n^j(-1)^{j+1} \lambda^{j} \bar e_j$ and $M\bar e_j=\bar e_{j-1}$, for $j=2,\dots,n$. In the basis $(\bar e_j)$, the matrix $M$ has the following form
$\bar M_\lambda=\displaystyle\sum_{j=1}^{n} \bar k_j E_{j,1}+\displaystyle\sum_{j=1}^{n-1}E_{j,j+1},$
with $\bar k_j=\complement_n^j(-1)^{j+1} \lambda^{j}.$
If $\lambda=0$, then we are done. If $\lambda\neq 0,$ we use a direct approach by looking for the coefficients $p_{i,j}$ of a matrix $P$ such that  $\bar M_\lambda P=PM_\lambda.$
Using the explicit expressions 
 of $PM_\lambda$ and 
of  $\bar M_\lambda P$
and by a direct identification of the coefficients, we get 

$p_{k,r}=\sum\limits_{j=0}^{r-1} (-1)^{k-r+j} \lambda^{k-r+j} \complement_{n-r+j}^{k-r+j} p_{1,j+1}$,  

\noindent
where  the numbers $ p_{1,j+1}$, $j=0,\dots,n-1,$ are seen as parameters.
In particular, the matrix $P$ whose coefficients in the above basis ($\bar e_s$) are

$p_{ij}=\left\{\begin{array}{cll}
(-1)^{i-j}\lambda^{i-j}\complement_{n-j}^{i-j} &\text{ if }& i\ge j, \\
0 &\text{ if }& i<j
\end{array}\right.
$

\noindent
is solution, where $\complement_p^q=\frac{p!}{q!(p-q)!}$, $p\ge q$.
\end{proof}
\begin{definition}[Notation]\label{def:D0n}
When $M$  is an $n\times n$ nonderogatory  real matrix with a unique real eigenvalue $\lambda$ of multiplicity $n$, we denote by $\mathfrak{D}_\lambda^n$, the corresponding 2-step solvable Lie algebra $\mathcal G_M.$
\end{definition}
\begin{lemma}\label{lem:multiplerealeigenvalues}Suppose $M$ is a nonderogatory  $n\times n$ real matrix  with $p$ distinct eigenvalues $\lambda_1, \cdots,\lambda_p,$  all of which are real and of respective multiplicity $k_1,\cdots, k_p$, where
$k_1+\cdots+ k_p=n.$ Then the corresponding  Lie algebra $\mathcal G_M$ is isomorphic to the direct sum 
$ \mathfrak{D}_{\lambda_1}^{k_1}\oplus \cdots \oplus \mathfrak{D}_{\lambda_p}^{k_p}$ of the Lie algebras $\mathfrak{D}_{\lambda_i}^{k_i},$ $i=1,\dots,p.$
\end{lemma}
\begin{proof} Lemma \ref{lem:splitting} to 
Equality (\ref{eq:split}), where $\ker\chi_j(M)=\ker(M-\lambda_j \mathbb I_{\mathbb R^n})^{k_j} =:\mathcal E_j,$ $j=1,\dots,p,$  leads to $\mathcal G_{M} = \mathcal G_{M_1}\oplus\dots\oplus\mathcal G_{M_p}$, where $M_j$ is the restriction of $M$ to $\mathcal E_j.$
 Taking Lemma \ref{lem:nonderogatorysubmatrices} into account, each $M_j$ is a nonderogatory $k_j\times k_j$  real matrix with a unique eigenvalue $\lambda_j$ of multiplicity $k_j$, so that $\mathcal G_{M_j}$ is exactly $\mathfrak{D}_{\lambda_j}^{k_j},$ by Definition \ref{def:D0n}., and thus
$\mathcal G_M= \mathfrak{D}_{\lambda_1}^{k_1}\oplus \dots \oplus \mathfrak{D}_{\lambda_p}^{k_p}$.
\end{proof}

\begin{lemma}\label{lem:uniquerealeigenvalue}
If a real $n\times n$ matrix $M$ admits a unique real eigenvalue $\lambda$, then the Lie algebra $\mathcal G_M =:\mathfrak{D}_{\lambda}^{n} $ is isomorphic to ${\mathfrak D}_{0}^{n}$ as in Example \ref{ex:D0n} \;.
\end{lemma}
\proof
Following Lemma \ref{lem:uniquerealeigenvalueJordanform}, consider a basis in which $M$ has the following form
$M_\lambda=\lambda\mathbb I_{\mathbb R^n}+\displaystyle\sum_{i=1}^{n-1}E_{i,i+1}$ and set $M_0:=\displaystyle\sum_{i=1}^{n-1}E_{i,i+1}.$
Note that $M_0$ is given in the same basis as $M_\lambda$ and  it has the same form as the one  in Example 
\ref{ex:D0n}. 
 One notes that the two nonderogatory matrices $M_0$ and $M_\lambda$ commute. More precisely,  $M_\lambda$ is a polynomial in  $M_0$ as it 
 can be written as $M_\lambda=\lambda( M_0)^0+M_0.$ Hence $C(M_0)=C(M_\lambda)=\mathbb K[M_0]$.  Thus we have  ${\mathfrak D}_{0}^{n} = \mathcal G_{M_0} = \mathcal G_{M_\lambda} =\mathfrak{D}_{\lambda}^{n}  .$
\qed

Lemmas \ref{lem:multiplerealeigenvalues} and \ref{lem:uniquerealeigenvalue} prove the following.
\begin{proposition}\label{prop:classification-multiple-eigenvalues}
Suppose $M$ is a nonderogatory  $n\times n$ real matrix  with $p$ distinct eigenvalues $\lambda_1, \cdots,\lambda_p,$  all of which are real and of respective multiplicity $k_1,\cdots, k_p$, where
$k_1+\cdots+ k_p=n.$ Then the Lie algebra $\mathcal G_M$ is isomorphic to the direct sum $ \mathfrak{D}_{0}^{k_1}\oplus \cdots \oplus \mathfrak{D}_{0}^{k_p}$ of the Lie algebras $\mathfrak{D}_{0}^{k_j},$ $j=1,\dots,p.$
\end{proposition}
\begin{proof}Indeed, if $M$ is a nonderogatory  $n\times n$ real matrix  with $p$ real eigenvalues $\lambda_1, \cdots,\lambda_p,$ 
 of respective multiplicity $k_1,\cdots, k_p$, where
$k_1+\cdots+ k_p=n,$ then Lemma \ref{lem:multiplerealeigenvalues} ensures that $\mathcal G_M$ is isomorphic to the direct sum $ \mathfrak{D}_{\lambda_1}^{k_1}\oplus \cdots \oplus \mathfrak{D}_{\lambda_p}^{k_p}$ and Lemma  \ref{lem:uniquerealeigenvalue} further proves that each $\mathfrak{D}_{\lambda_j}^{k_j}$ is isomorphic to $\mathfrak{D}_{0}^{k_j}$, for  $j=1,\dots,p.$ 
\end{proof}
This concludes the proof of Theorem \ref{thm:classification} in the case where $M$ is a nonderogatory  $n\times n$ real matrix  with only real eigenvalues.

\subsection{Nonderogatory matrices all of whose eigenvalues are complex}\label{Sect:complexeigenvalues}

\subsubsection{Real matrices diagonalizable in $\mathbb C$}
Let $M$ be a nonderogatory $n\times n$ real matrix all of whose eigenvalues are complex (and not real),  in this case  $n$ is even. We suppose that $M$ is diagonalizable in $\mathbb C$. Thus, $M$ being nonderogatory imposes that all the eigenvalues are of multiplicity $1$ and hence pair-wise distinct. That is, $M$ has $n$ distinct (complex) eigenvalues $\lambda_j,$ $\bar \lambda_j,$ $j=1,\dots,\frac{n}{2}.$ Let $\lambda =\lambda_R-i\lambda_I$  be an eigenvalue of $M$,  where $\lambda_R,\lambda_I$ are real numbers and $\lambda_I\neq 0$.
 Consider an  eigenvector  $v$ (with complex components) of $M$ with corresponding eigenvalue $\lambda$. Further set $\mathcal E_\lambda:=
\{ z v+\overline{zv}, z\in \mathbb C\}.$  Note that, as $\bar v$ is also an eigenvector of $M$ with eigenvalue $\bar \lambda$, we also have $\mathcal E_\lambda=\mathcal E_{\bar\lambda}.$ Furthermore,  $\mathcal E_\lambda$ is a real 2-dimensional vector subspace of $\mathbb R^n$ 
which is stable by $M$ and 
 the vectors $Re(v)$ and $Im(v)$ form a basis of  $\mathcal E_\lambda$ in which the matrix  of the restriction of $M$ is
 $M_\lambda:=\begin{pmatrix}\lambda_R&-\lambda_I\\
\lambda_I &\lambda_R\end{pmatrix}.$ 
Indeed, any element of  $\mathcal E_\lambda$ is of the form 
$  z v+\overline{zv} = 2Re(z)\; Re(v) -2 Im(z)\; Im(v) $
and taking the real and imaginary parts of both sides of the equation 
\begin{eqnarray}Mv&=&\lambda v = \Big(\lambda_R-i\;\lambda_I\Big)\; \Big(Re(v)+i\;Im(v)\Big) \nonumber\\
&=& \Big(\lambda_R\; Re(v) + \lambda_I\; Im(v)\Big) + i \;\Big(\lambda_R\; Im(v) -\lambda_I \;Re(v)\Big), \end{eqnarray}
 one gets 
\begin{eqnarray} 
 M Re(v)=\lambda_R\; Re(v) + \lambda_I\; Im(v), \nonumber\\
 M Im(v)=- \lambda_I\; Re(v) + \lambda_R\; Im(v). 
\end{eqnarray} 
Note that  $M_\lambda$ is a nonderogatory $2\times 2$ real matrix whose characteristic polynomial is $\chi_\lambda(X)=(X-\lambda_R)^2+\lambda_I^2$ and both $Re(v)$ and $Im(v)$ are in $\ker\chi_\lambda(M) .$  We denote by $\mathcal A_{\lambda}= \mathbb R \mathbb I_{\mathcal E_\lambda}\oplus \mathbb R  M_\lambda$ the 2-dimensional real vector space generated by the endomorphisms $\mathbb  I_{\mathcal E_\lambda}$ and $M_\lambda$ of $\mathcal E_\lambda.$ So $\mathbb R[M_\lambda]=\mathcal A_{\lambda}$ and the $4$-dimensional space $\mathcal A_{\lambda}\oplus \mathcal E_\lambda$, inherits the Lie bracket of the semi-direct sum  $\mathcal  G_{{ M}_{\lambda}}:=\mathcal A_{\lambda}\ltimes \mathcal E_\lambda$. More precisely,  in the basis 
$e_1=\mathbb I_{\mathcal E_\lambda}= M_\lambda^0,$ $e_2=M_\lambda,$ $e_3=Re(v),$ $e_4=Im(v)$,  such a Lie bracket  reads
$[e_1,e_3] = e_3,$ $[e_1,e_4] = e_4,$ $[e_2,e_3] = Re(\lambda ) e_3 +  Im(\lambda) e_4$ and  $[e_2,e_3] = -  Im(\lambda) e_3 +Re(\lambda) e_4.$ 
The following holds.
 \begin{lemma}\label{lem:multiple-complex-eigenvalues}
Let $M$ be a nonderogatory $n\times n$ real matrix all of whose eigenvalues, say $\lambda_j, \bar \lambda_j,$ $j=1,\dots,q,$ are complex. Suppose $M$ is diagonalizable in $\mathbb C$, in which case  the eigenvalues are pairwise distinct and $2q=n.$ For each couple $\lambda_j, \bar \lambda_j  $, let  $\mathcal  G_{{\tilde M}_{\lambda_j}}$ be the 4-dimensional Lie algebra with Lie bracket 
$[e_1,e_3] = e_3,$ $[e_1,e_4] = e_4,$ $[e_2,e_3] = Re(\lambda_j) e_3 +  Im(\lambda_j) e_4$,  $[e_2,e_3] = -  Im(\lambda_j) e_3 +Re(\lambda_j) e_4.$ Then the Lie algebra  $\mathcal  G_{M}$ is isomorphic to the direct sum
 \begin{eqnarray} 
\mathcal  G_{M}= \mathcal  G_{{M}_{\lambda_1}}\oplus\dots\oplus\mathcal  G_{{M}_{\lambda_q}}.
\end{eqnarray} 
 \end{lemma}
\begin{proof}
 As above, consider the 2-dimensional real subspaces $\mathcal E_{\lambda_j}$. Note that, in the basis $(Re(v_j), Im(v_j)),$ $j=1,\dots,q$ of $\mathbb R^n,$ we have  
 the Jordan form of $M$ and $M=P$ diag($M_{\lambda_1},\dots,M_{\lambda_q}$) $P^{-1}$, 
where $P$ is the bloc diagonal matrix $P=diag(P_1,\dots,P_q)$ each $P_j$ being a $2\times 2$ real matrix whose first and second columns are the vectors $Re(v_j)$ and $Im(v_j)$ respectively.
The result follows by applying Lemma \ref{lem:splitting} to the direct sum  $\mathbb R^n=\mathcal E_{\lambda_1}\oplus\dots\oplus \mathcal E_{\lambda_q}$, where all the $\mathcal E_{\lambda_1}$ are stable by $M$ as explained above.
 Thus, $\mathcal G_M$ splits as a direct sum $\mathcal G_M = \mathcal  G_{{M}_{\lambda_1}}\oplus\cdots\oplus \mathcal  G_{{M}_{\lambda_q}},$  of the ideals $\mathcal  G_{{M}_{\lambda_j}}$, $j=1,\dots,q$.
 \end{proof}
Of course, each  Lie ideal $\mathcal G_{{M}_{\lambda_j}}$ is isomorphic to $\mathfrak{aff}(\mathbb C).$  This can be easily deduced from the classification of $4$-dimensional Lie algebras. For self-containedness purposes, we supply here an easy proof. Identifying $\mathcal E_{\lambda_j}$ with $\mathbb R^2$, we have
 \begin{lemma}\label{lem:isomorphism4D} Let $M_\lambda$ be a $2\times 2$ real matrix with complex eigenvalues $\lambda, \;\bar\lambda,$ where $\lambda =\lambda_R-i\lambda_I$  and $\lambda_R, \;\lambda_I$ are real numbers, with $\lambda_I\neq 0$ . Then $\mathcal G_{M_\lambda}$ is isomorphic to $\mathfrak{aff}(\mathbb C),$ as in Example \ref{ex:aff(C)}.\end{lemma}
\begin{proof}
Without loss of generality, we consider $M_\lambda=\begin{pmatrix}\lambda_R&-\lambda_I\\
\lambda_I &\lambda_R\end{pmatrix}$ in the canonical basis $(\tilde e_1\;,\; \tilde e_2)$ of $\mathbb R^2.$ In the basis $e_1' =\mathbb I_{\mathbb R^2},\; e_2'= M_\lambda$, $e_3'=\tilde e_1,\; e_4'=\tilde e_2$, the nonzero  Lie brackets of $\mathcal G_{M_\lambda}$ are
 $[e_1',e_3'] = e_3'$, $[e_1',e_4'] = e_4'$, $[e_2',e_3'] = \lambda_Re_3'+\lambda_I e_4'$,  $[e_2',e_4'] =- \lambda_Ie_3'+\lambda_R e_4'$. Now in the new  basis $X_1:=e_1'$, $X_2:=-\frac{\lambda_R}{ \lambda_I}e_1' + \frac{1}{ \lambda_I}e_2'$, $X_3:=pe_3' -qe_4'$, $X_4:=qe_3' + pe_4'$, with $p^2+q^2\neq 0$, the Lie bracket reads
$[X_1,X_3]=X_3$, $[X_1,X_4]=X_4$, $[X_2,X_3]=X_4$, $[X_2,X_4]=-X_3$ as for that of  $\mathfrak{aff}(\mathbb C).$ In other words, the invertible linear map $\phi : \mathfrak{aff}(\mathbb C) \to  \mathcal G_{M_\lambda}$, $\phi(e_j)=X_j$, $j=1,2,3,4,$ is an isomorphism between the Lie algebras $\mathfrak{aff}(\mathbb C)$ and $  \mathcal G_{M_\lambda}$.
  \end{proof}

\begin{proposition}\label{prop:classification-multiple-complexeigenvalues}
Suppose a nonderogatory  $n\times n$ real matrix $M$ has $2q=n$  distinct complex eigenvalues $\lambda_j,\bar\lambda_j,$  $j=1\dots, q.$ Then the Lie algebra $\mathcal G_M$ is isomorphic to the direct sum $\mathfrak{aff}(\mathbb C)\oplus \cdots \oplus \mathfrak{aff}(\mathbb C)$ of  $q$ copies of the Lie algebra $\mathfrak{aff}(\mathbb C)$.
\end{proposition}
\begin{proof}Suppose $M$ is a nonderogatory  $n\times n$ real matrix  with $2q$ distinct complex (nonreal) eigenvalues $\lambda_j,\bar\lambda_j,$ $j=1\dots, q.$  From Lemma \ref{lem:multiple-complex-eigenvalues}, the Lie algebra $\mathcal G_M$ is isomorphic to the direct sum $\mathcal G_M = \mathcal  G_{{M}_{\lambda_1}}\oplus\cdots\oplus \mathcal  G_{{M}_{\lambda_q}},$  of the ideals $\mathcal  G_{{ M}_{\lambda_j}}$, $j=1,\dots,q$ and according to Lemma  \ref{lem:isomorphism4D}, each  $\mathcal  G_{{M}_{\lambda_j}}$, is isomorphic to $\mathfrak{aff}(\mathbb C)$. 
\end{proof}

\subsubsection{Example: the circular permutation of the  vectors of a basis}\label{Sect:exple-complexeigenvalues}
Here is a typical example of a (real) nonderogatory linear map with $n$ real and complex eigenvalues, hence diagonalizable in $\mathbb C$. 
Let  $\psi:\mathbb R^n\to\mathbb R^n$ be the linear map given by the circular permutation of the canonical basis $\psi(\tilde e_i) = \tilde e_{i+1}$, for $i=1,\dots,n-1$ and $\psi(\tilde e_{n})=\tilde e_1$. The vector $\tilde e_1 $ is such that $(\tilde e_1,\psi(\tilde e_1), \dots, \psi^{n-1}(\tilde e_1))$  coincides with the basis  $(\tilde e_1, \dots, \tilde e_n)$. As a matter of fact, any vector $\tilde e_i$ of this basis is such that   $(\tilde e_i,\psi(\tilde e_i), \dots, \psi^{n-1}(\tilde e_i))$ is again a basis of $\mathbb R^n$. The  map $\psi$ is non-derogatory and its matrix in the above basis reads $[\psi]=E_{1,n}+\displaystyle\sum_{i=1}^{n-1}E_{i+1,i}.$
Its characteristic polynomial is $\chi(X)=X^n-1$, up to a sign. So the eigenvalues are the complex nth roots of $1.$ They are $z_k=e^{ik\frac{2\pi}{n}}$, where $k=1,2,\dots,n.$
When $n=2,$ then it reads  $[\psi]=E_{1,2}+E_{2,1}$ and has the two distinct real eigenvalues $z_1=-1$ and $z_2=1.$  So $\mathcal G_\psi$ is isomorphic to 
$\mathfrak{aff}(\mathbb R)\oplus \mathfrak{aff}(\mathbb R).$ For $n=3,$ the eigenvalues are $z_1=1,$ $ z_2=-e^{i\frac{\pi}{3}}$, 
$ z_3=-e^{-i\frac{\pi}{3}}.$ So  $\mathcal G_\psi$ is isomorphic to $\mathfrak{aff}(\mathbb R)\oplus\mathfrak{aff}(\mathbb C)$. For $n=4,$ the eigenvalues are $z_1=-1,$ $z_2= 1,$ $z_3= i,$ $z_4= -i$, thus  $\mathcal G_\psi$ is isomorphic to $\mathfrak{aff}(\mathbb R)\oplus \mathfrak{aff}(\mathbb R)\oplus\mathfrak{aff}(\mathbb C).$  When $n=5$, there are one real and $4$ complex eigenvalues and hence we get
 $\mathcal G_\psi=\mathfrak{aff}(\mathbb R)\oplus\mathfrak{aff}(\mathbb C)\oplus\mathfrak{aff}(\mathbb C).$
For $n=7$, we have one real  and six complex, namely  $1, -e^{i\frac{\pi}{7}},  -e^{-i\frac{\pi}{7}}, e^{i\frac{2\pi}{7}}, e^{-i\frac{2\pi}{7}},  -e^{i\frac{3\pi}{7}},- e^{-i\frac{3\pi}{7}}.$
Hence we have $\mathcal G_\psi=\mathfrak{aff}(\mathbb R)\oplus\mathfrak{aff}(\mathbb C)\oplus\mathfrak{aff}(\mathbb C) \oplus\mathfrak{aff}(\mathbb C).$


\subsubsection{Real matrices non-diagonalizable in $\mathbb C$}\label{sect:nondiagonalizableinC}

In this section, we discuss the case of $n\times n$ nonderogatory real matrices $M$ all of whose eigenvalues are complex (nonreal), 
but which are not diagonalizable in $\mathbb C.$  In this case, $n$ is even.
Lemma \ref{lem:uniqueeigenvalueJordanform} is a well known result (attributed to Hirsch and Smale).
\begin{lemma}[Jordan form]\label{lem:uniqueeigenvalueJordanform}
Suppose  a nonderogatory  $n\times n$ real matrix $M$ has only two  eigenvalues which are both complex $\lambda=r+is$ and $\bar\lambda.$ If $M$ is not diagonalizable in $\mathbb C$, then there exists a basis $(\tilde e_1, \cdots, \tilde e_n)$ of $\mathbb R^n$ in which $M$ has the form
\begin{eqnarray}\label{jordancomplex}M_\lambda = \begin{pmatrix}M_{r,s}  & \mathbb I_2 & {\mathbf 0} &{\mathbf 0} &\cdots & {\mathbf 0}\\ {\mathbf 0} &M_{r,s}  &\mathbb I_2 & {\mathbf 0} &\cdots& {\mathbf 0}\\
 \vdots & \vdots  &\ddots & \ddots & \ddots& \vdots
\\
{\mathbf 0} & {\mathbf 0}&  \cdots &M_{r,s}  & \mathbb I_2 &0
\\
{\mathbf 0} & {\mathbf 0}&  \cdots &0&M_{r,s}  & \mathbb I_2 
\\
 {\mathbf 0} &{\mathbf 0}& \cdots  & {\mathbf 0} & {\mathbf 0} &M_{r,s}   \end{pmatrix},\end{eqnarray} 
where $M_{r,s}: =\begin{pmatrix} r&-s\\s&r\end{pmatrix}$,  $\mathbb I_2: =\begin{pmatrix} 1&0\\ 0&1\end{pmatrix}$  and ${\mathbf 0}: =\begin{pmatrix} 0&0\\ 0&0\end{pmatrix}$.
\end{lemma}
One easily sees that the characteristic and the minimal polynomials of $M_\lambda$ are both equal to $\chi(X)=
\Big((r-X)^2+s^2\Big)^{\frac{n}{2}},$ as in the factorization lemma \ref{lemma:factorization}. In the same basis in which $M_\lambda$ is in the form (\ref{jordancomplex}), consider the nonderogatory matrix
\begin{eqnarray}\label{jordancomplexcanonical}M_{0,1} = \begin{pmatrix}\tilde M_{0,1}  & \mathbb I_2 & {\mathbf 0} &{\mathbf 0} &\cdots & {\mathbf 0}\\ {\mathbf 0} &\tilde M_{0,1}  &\mathbb I_2 & {\mathbf 0} &\cdots& {\mathbf 0}\\
 \vdots & \vdots  &\ddots & \ddots & \ddots& \vdots
\\
{\mathbf 0} & {\mathbf 0}&  \cdots &\tilde M_{0,1}  & \mathbb I_2 &{\mathbf 0}
\\
{\mathbf 0} & {\mathbf 0}&  \cdots &{\mathbf 0}&\tilde M_{0,1}  & \mathbb I_2 
\\
 {\mathbf 0} &{\mathbf 0}& \cdots  & {\mathbf 0} & {\mathbf 0} &\tilde M_{0,1}   \end{pmatrix},\end{eqnarray}
where  $\tilde M_{0,1} =\begin{pmatrix} 0&-1\\ 1&0\end{pmatrix}$ and write $M_{0,1}=M_s+M_n$, where
\begin{eqnarray}
M_s=-\displaystyle\sum_{j=0}^{\frac{n}{2}-1}(E_{2j+1,2j+2}-E_{2j+2,2j+1}), \;\; M_n=\displaystyle\sum_{j=1}^{n-2}E_{j,j+2}.
\end{eqnarray}
We note that $M_s$ and $M_n$ commute $[M_s,M_n] =0$, more precisely, we have
\begin{eqnarray}M_sM_n&=& -\displaystyle\sum_{j=0}^{\frac{n}{2}-2}\Big(E_{2j+1,2j+4} - E_{2j+2,2j+3}\Big)= M_nM_s\;,
\end{eqnarray}
and so $[ M_{0,1},M_s]=[M_n ,M_s] = [ M_n,M_{0,1}]=0$. In particular,  $M_s$ and $M_n$ are respectively the 
semisimple and the nilpotent parts of $ M_{0,1}.$ The matrices $M_s$ and $M_n$ are both polynomials in $M_{0,1}$.
 Thus, the matrix  $M_\lambda=r\mathbb I_{\mathbb R^n}+sM_s+M_n$ is also polynomial in 
 $M_{0,1}$. This induces the following  equalities  
\begin{eqnarray}
\mathbb R[ M_{0,1}] = \mathbb R[ M_{\lambda}] \text{  and } \mathcal G_{M_{0,1}} =  \mathcal G_{M_\lambda}.
\end{eqnarray}
 We have thus proved the
\begin{proposition}\label{prop:isomorphismcomplexeigenvaluesnonsemiimple}
Let $M$ be  a real nonderogatory  $n\times n$ matrix. If $M$ has just two eigenvalues which are  complex and conjugate,  and if $M$ is non-diagonalizable in $\mathbb C$, then $  \mathcal G_{M}=\mathcal G_{M_{0,1}} .$ In this case, we write $\mathfrak{D}_{0,1}^n$  instead of $ \mathcal G_{M_{0,1}}.$
\end{proposition}
Finally, we deduce the following.
\begin{proposition}\label{prop:all-complex-eigenvalues-nondiagonalizable}
Let $M$ be  a real nonderogatory  $n\times n$ matrix. If all the eigenvalues $\lambda_1,\bar\lambda_1,\dots,\lambda_p,\bar\lambda_p$ of $M$ are  complex, then $  \mathcal G_{M}$  is  isomorphic to the direct sum $\mathfrak{D}_{0,1}^{k_1}\oplus\dots\oplus \mathfrak{D}_{0,1}^{k_p}$ of copies of $ \mathfrak{D}_{0,1}^{k_j}$, where $k_j$ is the multiplicity of the eigenvalues $\lambda_j$, $\bar \lambda_j$.
\end{proposition}
\begin{proof}
As above, applying the Primary Decomposition Theorem to Lemma \ref{lem:splitting} and further applying Lemma \ref{lem:splitting} together with Proposition \ref{prop:isomorphismcomplexeigenvaluesnonsemiimple}, yield the result.  
\end{proof}

\subsubsection{More on the Lie algebra $\mathfrak{D}_{0,1}^n$}
Without loss of generality (from Proposition \ref{prop:isomorphismcomplexeigenvaluesnonsemiimple}, Lie algebras of the form $\mathcal G_{M_\lambda}$ are all isomorphic, anyway), we suppose that the form (\ref{jordancomplexcanonical})   is in the canonical basis $(\tilde e_1,\dots,\tilde e_n)$ of $\mathbb R^n.$ Note that $\bar x =\tilde e_{n}$ is such that $(\bar x, M_{0,1} \bar x, \dots, M_{0,1}^{n-2} \bar x, M_{0,1}^{n-1}\bar x)$ is a basis of $\mathbb R^n.$ 
The matrices $M_s,$ $M_n$ satisfy $M_s^2=-\mathbb I_{\mathbb R^n}$ , $M_n^{\frac{n}{2}}=0$, $(M_s M_n)^{\frac{n}{2}}=0$ and for any $j=1,\dots, \frac{n}{2}-1,$
\begin{eqnarray}  \Big(M_n\Big)^{j}=\displaystyle\sum_{p=1}^{n-2j} E_{p,p+2j} \text{ and }
 M_s  \Big(M_n\Big)^{j}=-\displaystyle\sum_{p=0}^{\frac{n}{2}-j-1}\Big( E_{2p+1,2p+2j+2}-  E_{2p+2,2p+2j+1}\Big)
 .\nonumber\end{eqnarray} 
Note that each $ M_{0,1}^{p}  = (M_s+M_n)^p=\displaystyle\sum_{j=0}^p \complement_p^j \; (M_s)^j\; (M_n)^{p-j} $ , $p=0,1,2\dots,n-1$, is a linear combination of $\mathbb I_{\mathbb R^n}$, $M_s$, $(M_n)^{j},$ $M_s  (M_n)^{j}$, $j=1,\dots, \frac{n}{2}-1,$  where as above, $\complement_p^j=\frac{p!}{j!(p-j)!}$, $p\ge j$.  Thus $\Big(\mathbb I_{\mathbb R^n}, M_s, (M_n)^{j}, M_s  (M_n)^{j}, \; j=1,\dots, \frac{n}{2}-1
\Big)$ is another basis of the vector space underlying  $\mathbb R[M_{0,1}]$. 
 The codimension $2$ subspace  
$\mathcal N:=span\Big((M_n)^{j}, M_s  (M_n)^{j}, \; j=1,\dots, \frac{n}{2}-1
\Big)\ltimes \mathbb R^n$  is an ideal of $\mathfrak{D}_{0,1}^n$, for it  contains the derived ideal $\mathbb R^n =[\mathfrak{D}_{0,1}^n,\mathfrak{D}_{0,1}^n]$. The equalities $(M_n)^{\frac{n}{2}} = (M_s  M_n)^{\frac{n}{2}} =0$,  show that $\mathcal N$ is $\frac{n}{2}$-step nilpotent. We then deduce that  $\mathcal N$ is the nilradical of $\mathfrak{D}_{0,1}^n$.  Indeed, if we denote by $\mathfrak{F}:=\mathbb R \mathbb I_{\mathbb R^n} \oplus \mathbb R M_{0,1}$, the plane spanned by $ (M_{0,1})^0= \mathbb I_{\mathbb R^n}$ and $ M_{0,1}$, then the vector space underlying  $\mathfrak{D}_{0,1}^n$ splits as the direct sum $\mathfrak{D}_{0,1}^n = \mathfrak{F}\oplus \mathcal N$, so that, any subspace of dimension higher than $n-2$, must meet  $\mathfrak{F}$ non-trivially and hence must not be a nilpotent subalgebra of  $\mathfrak{D}_{0,1}^n$. Thus $\mathcal N$ is the biggest nilpotent ideal of $\mathfrak{D}_{0,1}^n$.
Altogether, $\mathfrak{D}_{0,1}^n$ is a nondecomposable, non-completely solvable (the adjoint of $M_s$ has complex eigenvalues) $2$-step solvable Frobenius Lie algebra with a codimension 2 non-Abelian nilradical $\mathcal N$.
One sees that $\mathfrak{D}_{0,1}^n$ is not isomorphic to any of the Lie algebras of the form $\mathcal G_M$ discussed so far. Indeed $\mathfrak{D}_0^n$ is completely solvable and has a codimension $1$ non-Abelian nilradical (except for $\mathfrak{aff}(\mathbb R)$ whose nilradical is Abelian). As for $\mathfrak{aff}(\mathbb C)$, it has an Abelian nilradical.

\subsection{Derivations and automorphisms of $\mathcal G_M$}
Let $M$ be a nonderogatory  $n\times n$ matrix with coefficients in a field $\mathbb K.$ The following describes the derivations of $\mathcal G_M.$
\begin{proposition}\label{prop:derivations}
Let $\mathfrak{N}$ stand for the normalizer of  $\mathbb K[M]$ in  $\mathcal M(n,\mathbb K),$ that is 
\begin{eqnarray}
\mathfrak{N} :=\{N\in\mathcal M(n,\mathbb K), \text{ such that } [ N,\mathbb K[M]]\subset \mathbb K[M]\},\end{eqnarray}
then the Lie algebra of derivations of   $\mathcal G_M$  is the semi-direct sum of $\mathfrak{N}$ and $\mathbb K^n$,
\begin{eqnarray}
Der(\mathcal G_M)=\mathfrak{N} \ltimes  \mathbb K^n.\end{eqnarray}
More precisely,  $D$ is a  derivation of $\mathcal G_M$ if and only if there exist a vector $x_D\in\mathbb K^n$ and a linear map $h:\mathbb K^n\to \mathbb K^n$ satisfying the condition $[h,M^j]\in \mathbb K[M],$  $ j=1,\dots,n-1,$ such that $D$
 is of the form
$D(a+x) =[h,a]+ ax_D+h(x)$
for every $a\in\mathbb K[M]$ and $x\in\mathbb K^n.$ \end{proposition}
\begin{proof}
Let D be a derivation of $\mathcal G_M.$
As the derived ideal $\mathbb K^n=[\mathcal G_M,\mathcal G_M]$ of $\mathcal G_M$ is preserved
by $D$, we can write $D(x)=h(x)$ and
$D(a)=D_{1,1}(a)+D_{1,2}(a),$  for every $a\in\mathbb K[M]$, ~ $x\in
\mathbb K^n,$ where $D_{1,1}:   \mathbb K[M]\to  \mathbb K[M]$, $D_{1,2}: \mathbb K[M]\to \mathbb K^n$, $h:\mathbb K^n\to\mathbb K^n$ are some linear maps.
Now, setting  $e_1:=\mathbb I_{\mathbb K^n}$ and  $x_D:=D_{12}(e_{1})$, the equality $D_{12}(a) =ax_D,$ for any $a\in \mathbb K[M]$, follows:
\begin{eqnarray}
0=D([e_{1}, a])=[D(e_{1}),a]+ [e_{1},D(a)] =
 [x_D,a] + [e_{1},D_{12}(a)]
=D_{12}(a)-ax_D\;. \nonumber
\end{eqnarray}
We also have
$h(ax) =  D[a,x]=
[D_{11}(a),x]+[a,h(x)] =D_{11}(a)x+ah(x)$, 
thus entailing $[h,a] = D_{11}(a) \in
\mathbb K[M],$ for any $a\in\mathbb K[M].$
So the equality $ D_{11}(a): = [h,a]$ stands as the definition of $D_{11} (a)$ and this
means that  the commutator $[h,a]$ of the linear map $ h$ and any 
$a$ must be a polynomial in $M.$ 
In conclusion, every derivation of $\mathcal G_M$ is given by a vector $x_D$  and a linear 
map $h: \mathbb K^n\to \mathbb K^n$ satisfying the condition $[ h,\mathbb K[M]]\subset \mathbb K[M].$
Let $\mathfrak{N}$ stand for the normalizer of  $\mathbb K[M]$ in  $\mathcal M(n,\mathbb K),$ that is 
\begin{eqnarray}
\mathfrak{N} :=\{N\in\mathcal M(n,\mathbb R), \text{ such that } [ N,\mathbb K[M]]\subset \mathbb K[M]\},\nonumber\end{eqnarray}
then the above means
$Der(\mathcal G_M)=\mathfrak{N} \ltimes  \mathbb K^n.$
\end{proof}
 Note that in Proposition \ref{prop:derivations}, the inner derivations of $\mathcal G_M$ are those for which the component $h$ is itself  a polynomial in $M$.
\subsubsection{Example : derivations of $\mathfrak{D}_0^n$}\label{ex:derivationsofD0n}
For $n=2,$
the normalizer of  $\mathbb R[E_{1,2}]$ in $\mathfrak{gl}(2,\mathbb R),$ is the 3-dimensional algebra  of matrices spanned by $E_{1,1},E_{1,2},E_{22}$, which is $1$ dimension higher than   $\mathbb R[E_{1,2}]$. For example, $E_{2,2}$ is not in  $\mathbb R[E_{1,2}]$  and will thus act as an outer derivation on $\mathfrak{D}_0^2.$ Considering elements of $\mathbb R^2$ as inner derivations, the space of derivations of 
$\mathfrak{D}_0^2,$ is $5$-dimensional and is spanned by $E_{1,1},E_{1,2},E_{22}, \tilde e_1,\tilde e_2.$ 
For $n=3,$
the normalizer of  $\mathbb R[E_{1,2}+E_{2,3}]$ in $\mathfrak{gl}(3,\mathbb R),$ is the $5$-dimensional algebra  of matrices spanned by $E_{1,1}-E_{3,3},$ $E_{22}+2E_{33}$, $E_{1,2}$, $E_{2,3}$, $E_{1,3}$.  For example,  $E_{1,1}-E_{3,3},$ $E_{22}+2E_{33}$, are not in  $\mathbb R[E_{1,2}+E_{2,3}]$, so they represent outer derivations of $\mathfrak{D}_0^3.$ Counting elements of $\mathbb R^3$ in as inner derivations, the space of derivations of 
$\mathfrak{D}_0^3,$ is $8$-dimensional.
More generally, for a given $n\ge 4$, if as above, we set $M_0:=E_{1,2}+E_{2,3}+\dots+E_{n-1,n}$, then the normalizer of $\mathbb R[M_0]$ in $\mathfrak{gl}(n,\mathbb R),$  is of dimension $2n-1$ and is spanned by 
$D_k:=E_{1,k}-\displaystyle\sum_{j=3}^{n-k+1}(j-2)E_{j,j+k-1}$, $D_{k}':=E_{2,k+1}+ \displaystyle\sum_{j=3}^{n-k+1}(j-1)E_{j,j+k-1}$, $k=1,\dots,n-2$ and $D_{n-1}:=E_{1,n-1}$, $D_{n-1}':=E_{2,n}$, $D_n:=E_{1,n}.$ For example, $D_1$ and $D_1'$ are not in $\mathbb R[M_0]$ and will act as non trivial outer derivations of $\mathfrak{D}_0^n$.  In particular $\mathbb R[M_{0}]$ is not a Cartan subalgebra of $\mathfrak{gl}(n,\mathbb R)$.
Taking into account elements of  $\mathbb R^n$ as inner derivations, the space of derivations of $\mathfrak{D}_0^n$ is thus of dimension $3n-1.$

\subsubsection{Example : on derivations of $\mathfrak{D}_{0,1}^n$}\label{ex:derivationsofD01n}
For $n=4$, consider the $n\times n$ matrix $M_{0,1}=E_{2,1} -E_{1,2}+E_{4,3}-E_{3,4}+E_{1,3} +E_{2,4} $ as in Example \ref{ex:D01n}. The normalizer of $\mathbb R[M_{0,1}]$ in $\mathfrak{gl}(4,\mathbb R),$ is of dimension $6$  and is spanned by the matrices  $E_{1,1}+E_{2,2},$ $E_{3,3}+E_{4,4},$ $E_{1,2}-E_{2,1},$ $E_{3,4}-E_{4,3},$  $E_{1,3}+E_{2,4},$ $E_{2,3}-E_{1,4}$. In particular the two matrices $E_{1,1}+E_{2,2},$ $E_{3,3}+E_{4,4},$ are not elements of $\mathbb R[M_{0,1}]$. So they both represent non-trivial outer derivations of $\mathfrak{D}_{0,1}^4$. More generally, for any $n\ge 4$, if we let again $M_{0,1}$ stand for the $n\times n$ matrix $M_{0,1}$, the normalizer of $\mathbb R[M_{0,1}]$ in $\mathfrak{gl}(n,\mathbb R),$ contains $\mathbb R[M_{0,1}]$ properly.  For example for $n\ge 6$, the $n\times n$ matrix   $
Z_1:=E_{1,1}+E_{2,2}-\displaystyle\sum_{j=2}^{\frac{n}{2}-1}(j-1)(E_{2j+1,2j+1}+E_{2j+2,2j+2})
$
is in the normalizer of $\mathbb R[M_{0,1}]$ in $\mathfrak{gl}(n,\mathbb R)$, but not in $\mathbb R[M_{0,1}]$.

 Indeed, for any $s,$ we have $  [Z_1, \Big(M_n\Big)^{s}] 
= s\displaystyle\displaystyle\sum_{j=1}^{n-2s}E_{j,j+2s} = s \Big(M_n\Big)^{s}$, $[Z_1, M_s] =0$ and  $[Z_1, M_s\Big(M_n\Big)^{s}] = M_s[Z_1, \Big(M_n\Big)^{s}] + [Z_1, M_s] \Big(M_n\Big)^{s}=  sM_s\Big(M_n\Big)^{s}$ . 

So the linear map $M\mapsto [Z_1,M]$ preserves $\mathbb R[M_{0,1}]$ and has the following diagonal matrix diag($0,0,1,2,\dots,\frac{n}{2}-1, 1,2,\dots,\frac{n}{2}-1$) in the following new basis
 $(\mathbb I_{\mathbb R^n}, M_s, M_n,(M_n)^2,\dots,(M_n)^{\frac{n}{2}-1}, M_s M_n, M_s(M_n)^2,\dots,M_s(M_n)^{\frac{n}{2}-1})$ and  $Z_1$ is not an element of  $\mathbb R[M_{0,1}]$, given that $\mathbb R[M_{0,1}]$ is Abelian. In particular $\mathbb R[M_{0,1}]$ is not a Cartan subalgebra of $\mathfrak{gl}(n,\mathbb R)$. Also, $Z_1$ will act as an outer derivation of $\mathfrak{D}_{0,1}^n$.

We have the following.
\begin{theorem}\label{prop:abeliannilrad}Let $M$ be an $n\times n$ nonderogatory real matrix. The following are equivalent. 

(1) Every derivation of $\mathcal G_M$ is an inner derivation.

(2) $\mathbb R[M]$ is a Cartan subalgebra of $\mathcal Gl(n,\mathbb R)$.

(3)  The matrix $M$ has $n$ distinct  (real or  complex) eigenvalues. 

(4) The nilradical of $\mathcal G_M$ is Abelian.

(5)  $\mathcal G_M$ is the direct sum $\mathcal G_M=\mathfrak{aff}(\mathbb R)\oplus\dots\oplus \mathfrak{aff}(\mathbb R)\oplus \mathfrak{aff}(\mathbb C)\oplus\dots\oplus \mathfrak{aff}(\mathbb C)$  

of only copies of $\mathfrak{aff}(\mathbb R)$ and $\mathfrak{aff}(\mathbb C)$.
\end{theorem}
\begin{proof}The equivalence between (1) and (2) is a direct consequence of Proposition \ref{prop:derivations}. Indeed, every derivation of $\mathcal G_M$ is inner, if and only if the normalizer $\mathfrak N$ of $\mathbb R[M]$ in  $\mathcal M(n,\mathbb R),$  coincides with $\mathbb R[M]$. 
The equivalence between (3) and (4) and (5) is shown as follows. 
From Theorem \ref{thm:classification}, the matrix $M$ has $n$ distincts eigenvalues, 
say $p$ distinct real eigenvalues and 2q distinct complex eigenvalues, with $p+2q=n$, if and only if $\mathcal G_M$ is isomorphic to the direct sum $\mathcal G_M=\mathfrak{aff}(\mathbb R)\oplus\dots\oplus \mathfrak{aff}(\mathbb R)\oplus \mathfrak{aff}(\mathbb C)\oplus\dots\oplus \mathfrak{aff}(\mathbb C)$  of $p$ copies of $\mathfrak{aff}(\mathbb R)$ and $q$ copies of $\mathfrak{aff}(\mathbb C)$.Thus (3) and (5) are equivalent.
Let us remind here that, by a direct sum, we mean that  each copy of either  $\mathfrak{aff}(\mathbb R)$ or $\mathfrak{aff}(\mathbb C)$, is an ideal of the Lie algebra $\mathcal G_M$. But of course, as  both $\mathfrak{aff}(\mathbb R)$ and  $\mathfrak{aff}(\mathbb C)$ have Abelian nilradical, then so does $\mathcal G_M.$ So (3) implies (4). Conversely, from the classification Theorem \ref{thm:classification}, the only Lie algebras of the form $\mathcal G_M$ that have an Abelian nilradical, with $M$ a nonderogatory matrix, are $\mathfrak{aff}(\mathbb R)$,  $\mathfrak{aff}(\mathbb C)$ and all Lie algebras made of direct sums of copies of them. Hence (4) implies (5) and thus (3) and (4) are equivelent, too.
Let us prove that (3) implies (1). As $\mathcal G_M$ has no center, each copy of either  $\mathfrak{aff}(\mathbb R)$ or $\mathfrak{aff}(\mathbb C)$, is preserved by every derivation of $\mathcal G_M$. So a derivation of $\mathcal G_M$ is an inner derivation if and only if its restriction to each copy of either  $\mathfrak{aff}(\mathbb R)$ or $\mathfrak{aff}(\mathbb C)$ is inner.
From \cite{diatta-manga}, for any $m\ge 1,$  the Lie algebra  $\mathfrak{aff}(m,\mathbb R)$ only has inner derivations.
On the other hand, as above, the $2\times 2$ nonderogatory matrix $M_{-1}=E_{2,1}-E_{1,2}$ satisfies $\mathcal G_{M_{-1}} = \mathfrak{aff}(\mathbb C) $. For a $2\times 2$ matrix $M$ with coefficients $m_{i,j}$ , the matrix $[M,M_{-1}] =(m_{2,2}-m_{1,1})(E_{1,2}+E_{2,1}) +(m_{1,2}+m_{2,1}) (E_{1,1}-E_{2,2})$ is in $\mathbb R [M_{-1}]$ if and only if $m_{1,2}=- m_{2,1}$ and  $m_{1,1}= m_{2,2}$, that is, if and only if $M$ itself lies in  $\mathbb R [M_{-1}]$. Hence all derivations of $ \mathfrak{aff}(\mathbb C)$ are inner derivations. In order to complete the proof, we need to show that (1) implies (5). First, recall that $\mathcal D_0^n$  and $\mathcal D_{0,1}^n$  both possess outer derivations, for any $n\ge 2,$ as seen in Examples \ref{ex:derivationsofD0n} and \ref{ex:derivationsofD01n}. So if  $M$ is a real nonderogatory matrix, in order for the Lie algebra $\mathcal G_M$ to only have inner derivations, it must not be isomorphic to $\mathcal D_0^n$ or  $\mathcal D_{0,1}^n$ if it is indecomposable and in the case where it is decomposable, it must not contain a copy of  $\mathcal D_0^p$ or $\mathcal D_{0,1}^p$, for any $p\ge 2$, as a component of its decomposition into a direct sum of ideals. Hence the matrix $M$ has $n$ distinct eigenvalues. Thus (1) implies (5).
\end{proof}
Note that Proposition \ref{prop:abeliannilrad} is in agreement with the classification of Cartan subalgebras of $\mathfrak{gl}(n,\mathbb R)$ supplied by Kostant \cite{kostant} and Sugiura \cite{sugiura}.

\begin{proposition}\label{automorphisms} Let $M$ be a nonderogatory real $n\times n$ matrix. Any Lie algebra automorphism of $\mathcal G_M$ is of the form
\begin{eqnarray}
\psi (a,x) =  (\phi\circ a \circ \phi^{-1}, \phi\circ a \circ \phi^{-1}(x_0) +\phi(x)  )\end{eqnarray}
 for any $(a,x)\in \mathbb R[M]\ltimes \mathbb R^n =:\mathcal G_M$,
for some $x_0\in \mathbb R^n$ and some invertible linear map
 $\phi :\; \mathbb R^n\to\mathbb R^n$ such that $\phi\circ M \circ \phi^{-1}\in\mathbb R[M].$ 
The inverse of $\psi$ is given by
\begin{eqnarray}
\psi^{-1}(a,x)  = \Big(\phi^{-1}\circ a\circ\phi \; , \; \phi^{-1}(x - ax_0) \Big).\end{eqnarray}
\end{proposition}
\begin{proof}As the derived ideal $[\mathcal G_M,\mathcal G_M]=\mathbb R^n$ is preserved by any automorphism $\psi$, we set $\psi(a)=\psi_{1,1}(a) +\psi_{1,2} (a)$ and $\psi(x)=\phi(x)$ for any $a\in\mathbb R[M]$ and $x\in\mathbb R^n$, where $\psi_{1,1}:\mathbb R[M]\to\mathbb R[M]$,  $\psi_{1,2}:\mathbb R[M]\to\mathbb R^n$,  $\phi:\mathbb R^n \to\mathbb R^n$ are linear maps, with $\psi_{1,1}$ and $\phi$ invertible.  From the equality $\phi(ax)=\psi([a,x])= [\psi_{1,1} (a)+\psi_{1,2}(a),\phi(x)] = \Big(\psi_{1,1} (a)\circ\phi\Big)(x)$, we deduce $\psi_{1,1}(a)=\phi \circ a \circ \phi^{-1}$. In particular $\phi \circ a \circ \phi^{-1}\in\mathbb R[M]$, for any $a\in\mathbb R[M]$, which is equivalent to  $\phi \circ M \circ \phi^{-1}\in\mathbb R[M]$. Taking $b=e_1$ and $x_0:=\psi_{1,2}(e_1)$ in the equality
 $0=\psi([a,b])= [\psi_{1,1} (a)+\psi_{1,2}(a),\psi_{1,1} (b)+\psi_{1,2}(b)]=\psi_{1,1} (a)\psi_{1,2}(b)- \psi_{1,1} (b) \psi_{1,2}(a)$ yields $\psi_{1,2} (a)=\phi\circ a\circ\phi^{-1}x_0$, for any $a\in\mathbb R[M].$
\end{proof}
\subsection{Proof of Theorem \ref{thm:cartansubalgebras}}\label{proofofCartan-subalgebras-sl(n,R)} 
From Proposition \ref{prop:abeliannilrad}, for a given $n$, the number of isomorphism classes of $2$-step solvable Frobenius Lie algebras 
of dimension $2n$  of the form $\mathcal G_M :=\mathbb R[M]\ltimes \mathbb R^n$, where $\mathbb R[M]$ is a Cartan subalgebra of $\mathfrak{gl}(n,\mathbb R),$  is
 exactly $[\frac{n}{2}]+1$. Indeed, one can look at $[\frac{n}{2}]+1$ as the number (counting from zero)
of possible copies of $\mathfrak{aff}(\mathbb C)$ that one can count in a decomposable Lie algebra 
containing only copies of either $\mathfrak{aff}(\mathbb C)$ or $\mathfrak{aff}(\mathbb R)$.
On the other hand, from e.g. \cite{sugiura}, there are exactly $[\frac{n}{2}]+1$ non-conjugate Cartan subalgebras of 
$\mathfrak{gl}(n,\mathbb R).$ So we have derived Theorem \ref{thm:cartansubalgebras}, in a simple and direct way. More precisely we prove it as follows.
\begin{proof} Let $\mathfrak{h}$ be a Cartan subalgebra of
 $\mathfrak{sl}(n,\mathbb R)$ with a $k$-dimensional toroidal part. From \cite{sugiura}, up to conjugacy under an element of the Weyl group, there is a basis ($\tilde e_1,\dots, \tilde e_n$), considered here as the canonical basis of $\mathbb R^n$, in which   $\mathfrak{h}$ is of the form
$\mathfrak{h}=\Big\{\begin{pmatrix} D_1&-D_2&{\mathbf 0}\\
D_2 &D_1&{\mathbf 0}\\
{\mathbf 0}&{\mathbf 0}&D_3
\end{pmatrix}\Big\}$, where $D_1=\text{diag}(h_1,\dots,h_{k})$, 
$D_2=\text{diag}(h_{k+1},\dots,h_{2k})$
, $D_3=\text{diag}(h_{2k+1},\dots,h_{n})$ with $ h_j\in\mathbb R, j=1,\dots,n$, so that $\mathfrak{h}$ is conjugate to the algebra 
$\mathfrak{h}'=\Big\{\text{diag}(D_1',\dots,D_k', D_3),$  with  
$D_j'=\begin{pmatrix} { h_j} & -h_{k+j} \\
h_{k+j} &{ h_j}
\end{pmatrix}$ and $ D_3=\text{diag}(h_{2k+1},\dots,h_{n}), 
 h_j\in\mathbb R, j=1,\dots,n\;\Big\}$,
 obtained by reordering the canonical basis of 
$\mathbb R^n$ into
 $(\tilde e_1, \tilde e_{k+1}, \tilde e_2, \tilde e_{k+2}, \dots,\tilde e_k,
 \tilde e_{2k},\tilde e_{2k+1},\tilde e_{2k+2},\dots,\tilde e_{n}).$ A matrix in $\mathfrak{h}'$ is
 nonderogatory if and only if $(h_i,h_{k+i}) \neq (h_j,h_{k+j})$, whenever $i\neq j$ and $h_{2k+s}\neq h_{2k+l}$ whenever $s\neq l.$ But the existence of a matrix (a regular element) satisfying these
conditions are guaranteed by the fact that $\mathfrak{h}$ is a Cartan subalgebra. Hence, there exists a nonderogatory matrix $M$ with the $n$ distinct eigenvalues $h_1+ih_{k+1}, h_1-ih_{k+1},\dots , h_k+ih_{2k}, h_k-ih_{2k},$ $h_{2k+1},\dots, h_{n}$, such that , up to a conjugation, $\mathbb I_{\mathbb R^n}\oplus\mathfrak{h}=\mathbb R[M]=\mathfrak{B}_{\mathfrak{h}}$. So (a) implies (b).   Furtheremore, Theorem \ref{thm:classification} ensures that $\mathfrak{B}_{\mathfrak{h}}\ltimes \mathbb R^n$ is isomorphic to the direct sum $\mathfrak{aff}(\mathbb C)\oplus\dots\oplus\mathfrak{aff}(\mathbb C)\oplus\mathfrak{aff}(\mathbb R)\oplus\dots\oplus\mathfrak{aff}(\mathbb R)$ of $k$ copies of $\mathfrak{aff}(\mathbb C)$ and $(n-k)$ of $\mathfrak{aff}(\mathbb R)$. In fact, the equivalence between (b) and (c) has already been proved  by Theorem \ref{thm:classification}. 
Conversely, suppose $M$ is a $n\times n$ nonderogatory  real matrix with $n$ distinct eigenvalues, 2k of which are complex. From Proposition \ref{prop:abeliannilrad}, $\mathbb R[M]$ is a Cartan subalgebra of $\mathfrak{gl} (n,\mathbb R)$ and so (b) implies (a). Furthermore, using the Primary Decomposition Theorem and results from Section \ref{Sect:complexeigenvalues}, then $\mathbb R[M]$ can be put in the form $\mathfrak{h}'$.
\end{proof}

\section{On low dimensional 2-step solvable Frobenius Lie algebras}\label{sect:6D}
Applying Theorem \ref{thm:classification} to low dimensions, we obtain the following. In dimension 2, there is a unique  such Lie algebra, namely $\mathfrak{D}_0^1=\mathfrak{aff}({\mathbb R})$.  In dimension $4,$ the  Lie algebras of the form $\mathcal G_M$ are $\mathfrak{D}_0^2$, $\mathfrak{aff}({\mathbb C})$
and $\mathfrak{aff}({\mathbb R})\oplus \mathfrak{aff}({\mathbb R})$. They correspond to the family PHC7 of \cite{4d-para-hypercomplex} 
and to S11, S8 and S10 respectively in \cite{snow}.
In dimension 6, they are  
$\mathfrak{D}_0^3$, $\mathfrak{D}_0^2 \oplus\mathfrak{aff}({\mathbb R})$, 
$\mathfrak{aff}({\mathbb C}) \oplus \mathfrak{aff}({\mathbb R})$
and 
$\mathfrak{aff}({\mathbb R})\oplus \mathfrak{aff}({\mathbb R}) \oplus \mathfrak{aff}({\mathbb R})$. 
Theorem \ref{thm:classification-D6} below provides, up to isomorphism, a complete list of all 2-step solvable Frobenius Lie algebras of dimension $4$ or $6.$

\begin{theorem} \label{thm:classification-D6} A 2-step solvable Frobenius Lie algebra of dimension $2n\leq 6$ is either of the form  $\mathcal G_\psi =\mathbb R[\psi]\ltimes \mathbb R^n$  for some nonderogatory endomorphism $\psi$ of $\mathbb R^n,$ or isomorphic to $\mathcal G_{3,1}$ as in Example \ref{example1-3D}. 
Consequently, a 2-step solvable Frobenius Lie algebra of dimension less or equal to $6$ is either isomorphic to
 $\mathfrak{aff}({\mathbb R})$, $\mathfrak{D}_0^2$, $\mathfrak{aff}({\mathbb C})$, $\mathfrak{D}_0^3$, $\mathcal G_{3,1}$, or to one of the direct sums 
$\mathfrak{aff}({\mathbb R})\oplus \mathfrak{aff}({\mathbb R})$,  
 $\mathfrak{D}_0^2 \oplus\mathfrak{aff}({\mathbb R})$, 
$\mathfrak{aff}({\mathbb C}) \oplus \mathfrak{aff}({\mathbb R})$, 
$\mathfrak{aff}({\mathbb R})\oplus \mathfrak{aff}({\mathbb R}) \oplus \mathfrak{aff}({\mathbb R})$ of their copies. 
\end{theorem}
\begin{proof} The case $n=1$ is trivial, as every 2-dimensional non-Abelian Lie algebra is isomorphic to $\mathfrak{aff}(\mathbb R) =\mathcal G_\psi$, where $\psi=\mathbb I_{\mathbb R}$, as in Example \ref{ex:aff(R)}.
 Let $\mathcal G $ be a 2-step solvable Lie algebra of dimension $2n$, with $2\leq n\leq 3.$ Write $\mathcal G=\mathcal A\ltimes \mathbb R^n$ where $\mathcal A$ is an $n$-dimensional Abelian subalgebra  of $\mathfrak{gl}(n,\mathbb R)$. Note that, when $n\le 3,$ then $[\frac{n^2}{4}]+1=n$, so that from Jacobson's theorem  \cite{jacobson-schur}, every $n$-dimensional Abelian subalgebra of $\mathfrak{gl}(n,\mathbb R)$ is a MASA. Furthermore, as $\mathcal A$ contains the identity matrix $\mathbb I_{\mathbb R^n}$, we use the decomposition $\mathcal A= \mathbb R \mathbb I_{\mathbb R^n} \oplus L$ of $\mathcal A$ into a direct sum of the line 
$\mathbb R \mathbb I_{\mathbb R^n} $ and a MASA  $L$ of $\mathfrak{sl}(n,\mathbb R)$. 
When $n=2,$ then $L$ is 1-dimensional, namely $L$ is the line $L=\mathbb R M$, for some nonzero element $M$ of $\mathfrak{sl} (2,\mathbb R)$. But every nonzero element of  $\mathfrak{sl} (2,\mathbb R)$ is nonderogatory. 
Thus, $\mathcal A=\mathbb R \mathbb I_{\mathbb R^2} \oplus \mathbb R M= \mathbb R [M]$ and  $\mathcal A\ltimes\mathbb R ^2 = \mathcal G_M$.
According to Theorem \ref{thm:classification}, a 4-dimensional Frobenius Lie algebra $\mathcal G_M$ given by a nonderogatory matrix $M$, is isomorphic to $\mathfrak{D}_0^2$, $\mathfrak{aff} (\mathbb C)$, or $\mathfrak{aff}(\mathbb R)\oplus\mathfrak{aff}(\mathbb R)$.
 In the case where $n=3,$ referring  to the classification list quoted in \cite{winternitz}, there are  six classes of non-mutually conjugate MASAs   of $\mathfrak{sl}(3,\mathbb R)$. With the same notation as in  \cite{winternitz}, we denote them by  $L_{2,i}$, $i=1,\dots,6.$
The first one is the Lie algebra $L_{2,1}$  of diagonal matrices of the form  diag($k_1+k_2, -k_1+k_2,-2k_2)$, where $k_1,k_2\in\mathbb R.$  Any matrix in  $L_{2,1}$ is a polynomial in the nonderogatory matrix   $S_{2,1}=$diag$(1,0,-1).$ More precisely  diag($k_1+k_2, -k_1+k_2,-2k_2)=(k_2-k_1)S_{2,1}^0 +\frac{1}{2}(k_1+3k_2) S_{2,1} +\frac{3}{2}(k_1-k_2) S_{2,1}^2$.  So $L_{2,1}$ is a 2-dimensional subalgebra of the algebra $\mathbb R[ S_{2,1}]$ of polynomials in $S_{2,1}$
and we have the equality $\mathcal A_{2,1}= \mathbb R \mathbb I_{\mathbb R^n} \oplus L_{2,1}= \mathbb R[ S_{2,1}].$ 
 Theorem \ref{thm:classification} asserts that, since  $S_{2,1}$ has three distinct eigenvalues,  the corresponding Lie algebra $\mathcal A_{2,1}\ltimes \mathbb R^3= \mathcal G_{ S_{2,1} }$ is isomorphic to the direct sum $ \mathfrak{aff}(\mathbb R)\oplus \mathfrak{aff}(\mathbb R)\oplus \mathfrak{aff}(\mathbb R).$
Every element of the second class $L_{2,2}:=\Big\{k_1(E_{1,1}+E_{2,2}-2E_{3,3}) +k_2(E_{1,2}-E_{2,1}), k_1, k_2\in\mathbb R \Big\},$
  is a nonderogatory matrix with one real eigenvalue, namely the real $-2k_1$ and two complex conjugate eigenvalues $k_1+ik_2$ and $k_1-ik_2,$ except when $k_2=0.$ 
For instance, let $ S_{2,2}:=E_{1,2}-E_{2,1}$be  the nonderogatory matrix in $L_{2,2}$ corresponding to $k_1=0, k_2=1$; Every matrix in  $L_{2,2}$ is of the form 
$-2k_1  S_{2,2}^0+ k_2  S_{2,2}-3k_1  S_{2,2}^2.$ So  $L_{2,2}$  is a 2-dimensional subalgebra of $\mathbb R[ S_{2,2}]$ and the space $\mathcal A_{2,2} = \mathbb R \mathbb I_{\mathbb R^3} \oplus L_{2,2} $ is equal to $\mathbb R[ S_{2,2}]$ and the Lie algebra $\mathcal A_{2,2}\ltimes \mathbb R^3$ is isomorphic to $\mathcal G_{ S_{2,2} }.$ 
Theorem \ref{thm:classification} ensures that such a Lie algebra is isomorphic to the direct sum $ \mathfrak{aff}(\mathbb C)\oplus \mathfrak{aff}(\mathbb R).$
Next, we note that every element of $L_{2,3}:=\Big\{k_1(E_{1,1}+E_{2,2}-2E_{3,3})+k_2E_{1,2}, k_1,k_2\in\mathbb R \Big\}$
  is of the form $k_1S_{2,3}^0+ k_2 S_{2,3} - (3k_1+k_2)S_{2,3}^2$ where $S_{2,3}$ stands for the nonderogatory matrix
$S_{2,3}:=E_{1,2}+E_{3,3}$ 
with the double eigenvalue $0$ and the simple eigenvalue 1. Hence, the algebra $L_{2,3}$ is a 2-dimensional subalgebra of 
 $\mathbb R[S_{2,3}]$ and thus, the equality $\mathcal A_{2,3}= \mathbb R \mathbb I_{\mathbb R^n} \oplus L_{2,3}=\mathbb R[ S_{2,3}]$ entails the needed equality 
 $\mathcal A_{2,3}\ltimes \mathbb R^3 = \mathcal G_{ S_{2,3} }.$ According to Theorem \ref{thm:classification}, the  latter Lie algebra is isomorphic to $\mathfrak{aff}(\mathbb R)\oplus \mathfrak{D}_0^2.$ 
As regards  $L_{2,5}:=\Big\{k_1E_{1,2}+k_2E_{1,3}, k_1,k_2\in\mathbb R\Big\}$, 
the space $\mathcal A_{2,5}=\mathbb R\mathbb I_{3}\oplus L_{2,5}$
 corresponds to $\mathfrak{B}_{3,1}$ in Examples \ref{example1-3D}, \ref{example-general-nD}, it 
contains no nonderogatory matrix and  the Lie algebra  $\mathcal A_{2,5}\ltimes\mathbb R^3=\mathfrak{B}_{3,1}\ltimes \mathbb R^3=\mathcal G_{3,1}$ is nondecomposable. 
Every element of the algebra  $L_{2,6}:=\Big\{k_1(E_{1,2}+E_{2,3})+k_2E_{1,3}, k_1,k_2\in\mathbb R\Big\}$  
is a polynomial in the nonderogatory matrix  $M_{0}=E_{1,2}+E_{2,3}$ 
 with  $0$  as a unique eigenvalue of multiplicity 3. Thus, as above, we conclude that $\mathcal A_{2,6}=\mathbb R[M_0]$, so that  $\mathcal A_{2,6}\ltimes \mathbb R^3 = \mathcal G_{ M_{0} } =\mathfrak{D}_0^3$  (see Theorem \ref{thm:classification}).  
It is easy to see that, for the remaining algebra  
 $L_{2,4}:=\Big\{k_{1,3}E_{1,3}+k_{2,3}E_{2,3},k _{1,3},k_{2,3}\in\mathbb R\Big\}$
  in the list in  \cite{winternitz},  if we set  $\mathcal A_{2,4}= \mathbb R \mathbb I_{\mathbb R^n} \oplus L_{2,4}$, then there is no Frobenius functional on  the Lie algebra $\mathcal A_{2,4} \ltimes \mathbb R^3$, as  every linear form $\alpha$ satisfies $(\partial\alpha)^3=0.$  
Indeed, in the basis $e_1:=\mathbb I_{\mathbb R^3}$, $e_2:=E_{1,3}$, $e_{3}:=E_{2,3}$, $e_{3+j} =\tilde e_j$, $j=1,2,3,$ the Lie bracket reads

\noindent
$[e_1,e_{3+j}]=e_{3+j}$\;, $j=1,2,3,$
$[e_2,e_{6}]=e_{4}$\,, $[e_3,e_{6}]=e_{5}$\;, \; 
so $\partial e_1^* = \partial e_2^* = \partial e_3^* = 0$\;,\; 
$\partial e_4^* = - e_1^*\wedge e_4^*  - e_2^*\wedge e_6^* $\; , \;   $\partial e_5^* = - e_1^*\wedge e_5^*-  e_3^*\wedge e_6^* $\; ,\; 
 $\partial e_6^* = - e_1^*\wedge e_6^* $   and 
for any $s_j\in\mathbb R$, $j=1,\dots,6$, setting $\alpha=s_1e_1^*+\dots+s_6e_6^*$, we have
\begin{eqnarray}\partial \alpha&=& \partial (s_4e_4^*+s_5 e_5^*+s_6 e_6^*) = -s_4( e_1^*\wedge e_4^*  + e_2^*\wedge e_6^*)-s_5( e_1^*\wedge e_5^*-  e_3^*\wedge e_6^*)\nonumber\\
&&-s_6e_1^*\wedge e_6^* =
- e_1^*\wedge (s_4 e_4^*  + s_5 e_5^* + s_6 e_6^*)   
-(s_4e_2^* + s_5 e_3^*) \wedge e_6^*\;.
\end{eqnarray}
This is of the form $\partial \alpha=e_1^*\wedge \beta+\gamma\wedge e_6^*$. Hence $(\partial \alpha)^3=0.$ See also Remark \ref{rem:masa-no-frobenius}.
 \end{proof}

\section{Some concluding remarks}\label{sect:remarks}
Our discussions in this paper relates to three problems,  the classification of $2$-step solvable Frobenius Lie algebras, the classification of maximal Abelian subalgebras (MASAs) of  the  Lie algebra of square matrices and the Gerstenhaber's Theorem. We have proposed a condition under which Gerstenhaber's Theorem is valid for any number of commuting matrices. In particular, under such a  condition, Gerstenhaber's Theorem is valid for a set of 3 commuting matrices. We have also shown that every 2-step solvable Frobenius  Lie algebra $\mathcal G$  is a semidirect sum  $\mathcal G=\mathfrak B\ltimes \mathbb R^n$ of two Abelian Lagrangian subalgebras $\mathfrak B$ and $\mathbb R^n,$ where 
 $ \mathbb R^n$   is the derived ideal  $\mathbb R^n=[\mathcal G,\mathcal G]$ and $\mathfrak B$ is an n-dimensional MASA of  $\mathfrak{gl}(n,\mathbb R)$, such that the action $\mathfrak B\times (\mathbb R^n)^* \to (\mathbb R^n)^*,$ $(a,f)\mapsto -f\circ a$,  has an open orbit. 
Yet interestingly, we have also shown that if an $n$-dimensional Abelian subalgebra $\mathfrak{B}$ of $n\times n$ matrices over $\mathbb K$, has an open orbit on   $(\mathbb K^n)^*,$ for the above action, then $\mathfrak{B}$ is a MASA of  $\mathfrak{gl}(n,\mathbb K).$ 
Moreover,  $\mathcal A\ltimes \mathbb R^n$  and  $\mathfrak B\ltimes \mathbb R^n$  are isomorphic if and only if   $\mathcal A$ and  $\mathfrak B$ are conjugate. That is, $\mathfrak B=\phi \mathcal A \phi^{-1} $, for some invertible linear map $\phi: \mathbb R^n\to\mathbb R^n$. So the classification of 2-step solvable Frobenius Lie algebras is equivalent to that of  n-dimensional MASAs of $\mathfrak{gl}(n,\mathbb R)$, for which the above action has some open orbit. We have also given a complete characterization of the corresponding LSA. New characterizations of Cartan subalgebras of $\mathfrak{sl}(n,\mathbb R)$ are also given.
In any dimension, we have performed a complete classification of all $2$-step solvable Frobenius Lie algebras corresponding to nonderogatory  real matrices. 
We have also discussed examples corresponding to $n$ non-isomorphic $2$-step solvable Frobenius Lie algebras which are not given by nonderogatory matrices, in any dimension $n$. 
In low dimensions, we have classified all $2$-step solvable Frobenius Lie algebras up to dimension $6$.

\end{document}